\documentclass[11pt, a4paper, reqno]{article}
\usepackage{amsmath,amsthm,amssymb}
\usepackage{graphicx} % Required for inserting images
\usepackage[margin=1.2in]{geometry}
\usepackage[colorlinks=true]{hyperref}
\usepackage{enumitem}
\usepackage{pifont}
\allowdisplaybreaks
\numberwithin{equation}{section}
\numberwithin{figure}{section}

\newtheorem{theorem}{Theorem}[section]
\newtheorem{lemma}{Lemma}[section]
\newtheorem{remark}{Remark}[section]
\newtheorem{corollary}{Corollary}[section]

\newtheorem{proposition}{Proposition}[section]

\makeatother

\usepackage{babel}

\newcommand{\R}{\mathbb R}

\begin{document}
	\title{Global well-posedness for nonlinear generalized Camassa-Holm equation}
	
	\author{Nes\.ibe Ayhan\\
		\textit{Department of Mathematics and Scientific Computing}\\
		\textit{University of Graz}\\
		\texttt{nesibe.ayhan@uni-graz.at}
		\and
		N\.ilay Duruk Mutluba\c{s}\\
		\textit{Faculty of Engineering and Natural Sciences}\\
		\textit{Sabanc\i\ University}\\
		\texttt{nilaydm@sabanciuniv.edu}
		\and
		Bao Quoc Tang\\
		\textit{Department of Mathematics and Scientific Computing}
		\textit{University of Graz}\\
		\texttt{quoc.tang@uni-graz.at}
	}
	\date{}
	%\author{Nes\.ibe Ayhan, N\.ilay Duruk Mutluba\c{s}, Bao Quoc Tang}
	
	%\address{Nes\.ibe Ayhan \hfill\break 
		%Department of Mathematics and Scientific Computing, University of Graz\hfill\break Mozartgasse 14, 8010 Graz, Austria}
	%\email{nesibe.ayhan@uni-graz.at}
	
	% \address{N\.ilay Duruk Mutluba\c{s}\hfill\break
		% Faculty of Engineering and Natural Sciences, Sabanc\i\ University\hfill\break 34956\ Istanbul, T\"urkiye}
	% \email{nilaydm@sabanciuniv.edu}
	
	% \address{Bao Quoc Tang \hfill\break 
		% Department of Mathematics and Scientific Computing, University of Graz\hfill\break Heinrichstrasse 36, 8010 Graz, Austria
		% }
	% \email{quoc.tang@uni-graz.at}
	
	% \keywords{Camassa-Holm equation; nonlinearity; dispersion; local well-posedness; global well-posedness; semigroup
		% theory; commutator estimate; conserved quantity }
	
	\maketitle
	
	\begin{abstract}
		We establish local and global well-posedness for the Cauchy problem of a generalized Camassa-Holm equation where orders of the momentum and the nonlinearity can be arbitrarily high.  More precisely, we  consider the equation
		\begin{equation*}
			m_t + m_x u^p + b m u^{p-1}u_x = -(g(u))_x + (b+1)u^p u_x, \quad m = (1-\partial_x^2)^k u,
		\end{equation*}
		where $p \geq 1$, $k \geq 1$ are arbitrary, $b$ is a real parameter, and $g(u)$ is a smooth function. %The standard Camassa-Holm equation corresponds to $k=1$, $p=1$, $b=2$, and $g(u)=0$.\\
		The local well-posedness is shown by using Kato's semigroup approach, where we treat the nonlinearity directly using commutator estimates and the fractional Leibniz rule without having to transform it in any specific differential form. This well-posedness is obtained in the phase space $H^s$ for $s > 2(k-1) + 3/2$, which is consistent with the results for the classical Camassa-Holm equation.  We also prove the global existence of solutions by obtaining conserved quantity and applying the same idea from our local theory.
	\end{abstract}
	
	% \tableofcontents
	
	\section{Introduction}
	The Camassa-Holm (CH) equation is a fundamental model in the study of nonlinear waves. It is usually written as
	\begin{equation}
		m_t + 2u_x m + u m_x = 0, \quad \text{with} \quad m = u - u_{xx}.
	\end{equation}
	This equation possesses a deeply rich mathematical structure, which includes special solutions called peakons and the phenomenon of wave breaking. It was introduced by Camassa and Holm \cite{camassa1993integrable} to model the unidirectional propagation of shallow water waves over a flat bottom. Over the years, there has been many studies this equation, see e.g. \cite{constantin2001scattering,constantin1998wave,constantin1999shallow,ji2022wave,rodriguez2001cauchy} and its various generalizations the literature, see e.g. \cite{coclite2023classical,daros2019instability,degasperis2002new,gui2008global,escher2006global,escher2007shock,gui2008global,ji2021wave,novikov2009generalizations} and references therein. 
	
	\medskip
	The first approach is to generalize the momentum density $m$. The standard CH equation uses $m=u-u_{xx}$, and the common higher order generalization is to consider $m = (I - \partial_x^2)^k$ for $k\ge 1$. Note that this generalization can be derived from an abstract Euler flow with a suitable $H^k$ metric, see e.g. McLachlan \& Zhang \cite{mclachlan2009well}. The analysis of such modified CH equation has been studied in many works, see e.g. \cite{mu2011well,tian2011global,tang2015well} for single equation or \cite{chen2017well} for coupled systems. There are also other generalizations for the Camassa-Holm equation. For instance, the $\mu-$Camassa-Holm equation uses the momentum density $m=\mu(u)-u_{xx}$, where $\mu$ is the mean of $u$, which was introduced in \cite{khesin2008generalized} and has been studied in other works like \cite{fu2012blow} and \cite{Yamane2020-nk}. Another example is the $\alpha-$Camassa-Holm equation, which uses $m=u-\alpha^2u_{xx}$, where $\alpha$ is a constant. This was considered by Holm, Naraigh \& Tronci \cite{holm2009singular}. In a previous work D. Mutluba\c{s} \& Ayhan \cite{duruk2025local}, a general differential operator $L$ was used to study such dispersive generalizations.
	
	\medskip
	The second approach is to consider general nonlinearity in the equation, that is, instead of the term $2u_x m + u m_x$, which corresponds to $p=1$ in \eqref{gch}, one can study a more general form $bu^{p-1}u)x m + u^pm_x$ with an arbitrary power $p \geq 1$. This generalized family includes several notable integrable equations as particular cases, such as the Degasperis-Procesi equation (with $p=1$, $b=3$) \cite{constantin2009hydrodynamical} and the Novikov equation (with $p=2, b=3$) \cite{zhou2025persistence}. For the case of $k=1$ in the operator, Zhou \& Ji \cite{zhou2022well} established well-posedness and wave breaking for this class of equations with general $p\geq1$ using Kato's semigroup approach.

	\medskip
	Although there are many studies that consider these two directions of generalization, they have typically been developed separately. Our work fills this gap by combining these two directions. More precisely, we study the general case where a higher-order operator and a generalized nonlinearity are present at the same time, allowing us to analyze a much wider class of equations under a single framework. In this work, we investigate the Cauchy problem for the following generalized Camassa-Holm equation
	\begin{equation}
		m_t + m_x u^p + b m u^{p-1}u_x = - (g(u))_x + (b+1)u^p u_x,\quad m = (1-\partial_x^2)^k u
		\label{gch}
	\end{equation}
	where $p \geq 1$, $k \geq 1$ are arbitrary, $b$ is a real parameter, and $g(u)$ is a smooth function. 
	
	\medskip
	Our first objective is to prove local well-posedness for this general equation using Kato's semigroup theory:
	\begin{theorem}
		Let $u_0 \in H^s, s>2(k-1)+3/2$ be given. Then, there exists a maximal time of existence $T>0$, depending on $u_0$, such that there is a unique solution $u$ satisfying
		\begin{equation*}
			u \in C([0, T); H^s) \cap C^1([0, T); H^{s-1}).
		\end{equation*}
		Moreover, the solution $u(t,x)$ depends continuously on the initial data, i.e., the solution mapping
		\begin{equation*}
			\begin{aligned}
				\Psi: H^s &\mapsto C([0, T); H^s) \cap C^1([0, T); H^{s-1})\\
				u_0 &\mapsto u(\cdot, u_0)
			\end{aligned}
		\end{equation*}
		is continuous.
		\label{thm-local}
	\end{theorem}
	We prove this theorem by rewriting the equation in the quasi-linear form
	\begin{equation*}
		u_t+A(u)u=f(u),
	\end{equation*}
	then applying Kato's semigroup method. This is done by verifying key properties of the operator $A(u)$ and the nonlinearity $f(u)$. In many previous studies, a standard technique involved a tricky manipulation of the equation to express the nonlinear term $f(u)$ as a single spatial derivative of a certain $F$, $f(u)=\partial_x(F(u))$. This reformulation is highly effective, as it significantly simplifies the proof of the Lipschitz estimates needed for $f(u)$. For equations where either $k=1$ or $p=1$, such a manipulation is indeed possible. However, since we combine both $k>1$ and $p>1$, the algebraic structure becomes too complicated and it seems not possible to manipulate the terms into this convenient $\partial_x(F(u))$ form.
	
	\medskip
	Our work therefore takes a different approach.  The core of our idea lies in breaking down these terms directly and using the fractional Leibniz rule and some commutator estimates. This technique is particularly helpful because it allows us to eliminate the highest-order derivatives that would otherwise be problematic. This strategy of using commutators to handle higher order derivative terms was used in the recent work of D. Mutluba\c{s} \& Ayhan \cite{duruk2025local}, which we extend here to address the more general structure. Through this direct and systematic approach, we establish all necessary bounds for the local well-posedness proof in the Sobolev space $H^s(\mathbb{R})$ for $s > 2(k-1) + 3/2$. For the standard CH equation where $k=1$, this regularity requirement reduces to $s > 3/2$, demonstrating that our result aligns naturally with the classical theory. Moreover, this result is also a clear improvement over the previous work \cite{duruk2025local}, which required the stronger regularity $s > 2(k-1) + 5/2$.\\
	
	Our second objective is to establish the global existence of solutions when $b = p+1$.
	\begin{theorem}
		Let $k\geq2, p\geq 1$, and $b=p+1$. Assume that $u_0 \in H^s(\R)$ for $s>2(k-1)+3/2$, and $m_0 \in L^2(\R)$. Then, the Cauchy problem for the generalized Camassa-Holm equation \eqref{gch} admits a unique global solution
		\begin{equation*}
			u \in C([0, \infty), H^s(\R)) \cap C^{1}([0, \infty), H^{s-1}(\R)),
		\end{equation*}
		and the solution depends continuously on the initial data.
		\label{global-wellposedness}
	\end{theorem}
	We achieve this by identifying a conserved quantity that provides a uniform control on the solution, which explains the requirement $p = b+1$. To extend this global persistence to higher regularity Sobolev spaces $H^s$ with 
	$s>k$, we utilize the idea from our local theory. The commutator estimates and fractional Leibniz rule allow us to control the growth of the $H^s$-norm, leading to a global bound via Gronwall's lemma.
	
	\medskip
	\textbf{This paper is organized as follows}: In Section \ref{local}, we give some preliminaries and obtain the local well-posedness of solutions by using Kato's semigroup theory. In Section \ref{global}, we establish some conservation laws and uniform bounds and obtain the global existence of solutions.
	
	\medskip
	\textbf{Notation}: In this paper, we use the following notation:
	\begin{itemize}
		\item The Sobolev spaces $L^p(\R)$ and $H^{s}(\R)$, $s\in \mathbb R$, are defined through the norms
		\begin{equation*}
			\|u\|_{L^p}:= \|u\|_{L^p(\R)} = \left(\int_{\R}|u(x)|^pdx\right)^{1/p} \quad \text{ for } \quad 1 \le p < \infty,
		\end{equation*}
		\begin{equation*}
			\|u\|_{L^\infty}:= \|u\|_{L^{\infty}(\R)}:= \text{ess sup}\{|f(x)|: x\in \R\},
		\end{equation*}
		and
		\begin{equation*}
			\|u\|_{H^s}:= \|u\|_{H^s(\R)}:= \|(1+|\xi|^2)^{s/2}\widehat{f}(\cdot)\|_{L^2(\R)}
		\end{equation*}
		where $\widehat{f}(\xi)$ is the one dimensional Fourier transform of $f$.
		\item We denote by $\Lambda:= (1 - \partial_x^2)^{1/2}$ and for $s\in \R$, we have the following norms equivalence
		\begin{equation*}
			\|u\|_{H^s} \sim \|\Lambda^{s} u\|_{L^2}.
		\end{equation*}
	\end{itemize}
	\section{Local well-posedness}\label{local}
	\subsection{Preliminaries}
	Consider the Cauchy problem for the quasi-linear equation of evolution:
	\begin{equation}
		\frac{du}{dt}+A(u)u=f(u), \quad t\geq0, \quad u(0)=u_{0}.
		\label{general-form}
	\end{equation}
	Let $X$ and $Y$ be Hilbert spaces such that $Y$ is continuously and densely embedded in $X$, and let $Q: Y \mapsto X$ be a topological isomorphism. We suppose the following assumptions on the operator $A$ and nonlinearity $f$:
	
	\medskip
	$(A1)$ %Let $B_{r}(0)$ be an open ball centered in the origin in $Y$ with radius $r>0$. 
	For each $u\in Y$, the linear operator $-A(u): X \mapsto X$ generates a $C_{0}$-semigroup $T(t)$ on $X$ which satisfies
	\begin{equation*}
		\|T(t)\|_{\mathcal{L}(X)} \leq e^{\beta t}, \forall t\in[0, \infty),
	\end{equation*}
	for a constant $\beta > 0$.
	
	\medskip
	$(A2)$ For any $u \in Y$, $A(u)$ is a bounded linear operator from $Y$ to $X$, i.e, $A(u) \in \mathcal{L}(Y,X)$ with
	\begin{equation*}
		\|(A(u)-A(v))w\|_{X} \leq \lambda_{1} \|u-v\|_{X}\|w\|_{Y}, \quad u,v,w \in X.
	\end{equation*}

	\medskip
	$(A3)$ For any $u\in Y$, there exists a bounded linear operator $B(u) \in \mathcal{L}(X)$ satisfying $B(u) = QA(u)Q^{-1}-A(u)$, where $B:Y \mapsto \mathcal{L}(X)$ is uniformly bounded on bounded sets in $Y$. Moreover,
	\begin{equation*}
		\|(B(u)-B(v))w\|_{X} \leq \lambda_{2}\|u-v\|_{Y}\|w\|_{X}, \quad u,v \in Y, w\in X.
	\end{equation*}
	
	\medskip
	$(A4)$ $f$ is uniformly bounded on bounded sets in $Y$. Moreover, the map $f: Y \mapsto Y$ is locally $X$-Lipschitz continuous in the sense that there exists a constant $\lambda_{3}>0$ such that
	\begin{equation*}
		\|f(u)-f(v)\|_{X} \leq \lambda_{3}\|u-v\|_{X},\ \ \forall u,v \in B_{r}(0) \subseteq Y,
	\end{equation*}
	and locally $Y$-Lipschitz continuous in the sense that there exists a constant $\lambda_{4}>0$ such that
	\begin{equation*}
		\|f(u)-f(v)\|_{Y} \leq \lambda_{4}\|u-v\|_{Y},\ \ \forall u,v \in B_{r}(0) \subseteq Y,
	\end{equation*}
	where $B_r(0)$ is the open ball centered at $0$ in $Y$.%, and $\lambda_{1}, \lambda_{2}, \lambda_{3},$ and $\lambda_{4}$ depend only on $r$.
	\begin{theorem}\label{Kato} (\textit{Kato's Semigroup Theorem})\\
		Assume that $(A1)-(A4)$ hold. For given $u_{0} \in Y$, there is a $T>0$, depending on $u_{0}$, and a unique solution $u$ to \eqref{general-form} such that
		\begin{equation*}
			u=u(\cdot,u_{0}) \in C([0, T), Y) \cap C^{1}([0, T), X).
		\end{equation*}
		Moreover, the solution depends continuously on the initial data, i.e, the map \mbox{$u_{0} \mapsto u(\cdot, u_{0})$} is continuous from $Y$ to $C([0, T), Y) \cap C^{1}([0, T), X)$.
		\label{kato}
	\end{theorem}
	The following lemmas also will be useful for our main results:
	\begin{lemma}(Fractional Leibniz Rule)\ \cite{kato2006quasi}
		Let $\beta, \alpha$ be real numbers such that $-\beta < \alpha \leq \beta$. Then
		\begin{align*}
			&\|fg\|_{H^\alpha} \leq c \|f\|_{H^\beta} \|g\|_{H^\alpha}, \quad \text{if}\quad \beta > 1/2,&\\
			&\|fg\|_{H^{\alpha + \beta - \frac{1}{2}}} \leq c \|f\|_{H^\beta} \|g\|_{H^\alpha}, \quad \text{if}\quad \beta < 1/2,
		\end{align*}
		where $c$ is a positive constant depending only on  $\alpha$ and $\beta$.
		\label{leibniz}
	\end{lemma}
	
	For an operator $\mathcal T$ and a given function $f$, we can define the commutator
	\begin{equation*}
		[\mathcal T, f]g = \mathcal{T}(fg) - f\mathcal{T}g.
	\end{equation*}
	Then we have the following commutator estimate.
	\begin{lemma}\cite{lannes2013water}
		Let $n > 0$, $\tilde{s} \geq 0$, and $3/2 < \tilde{s} + n \leq \sigma$. Then, for all $f \in H^{\sigma}$ and $g \in H^{\tilde{s}+n-1}$, one has
		\begin{equation*}
			\|[\Lambda^n, f]g\|_{H^{\tilde{s}}} \leq c \|f\|_{H^\sigma} \|g\|_{H^{\tilde{s}+n-1}},
		\end{equation*}
		where $\Lambda = (I-\partial_x^2)^{1/2}$ and $c$ is a constant which is independent of $f$ and $g$.
		\label{estimate}
	\end{lemma}
	
	\begin{lemma}\cite{zhou2022well}
		If $F \in C^{\infty}(\mathbb{R})$ with $F(0) = 0$, then for any $r > 1/2$, we have
		\begin{equation*}
			\|F(u)\|_{H^r} \leq \widetilde{F}(\|u\|_{L^\infty})\|u\|_{H^r}, \quad u \in H^r(\mathbb{R}),
		\end{equation*}
		where $\widetilde{F}$ is a monotone increasing function depending only on $F$ and $r$.
		\label{composition}
	\end{lemma}
	
	\begin{lemma}\cite{kato1988commutator}\label{est1}
		If $r>0$, then $H^r(\R)\ \cap L^{\infty}(\R)$ is an algebra. Moreover,
		\begin{equation}
			\|fg\|_{H^r} \leq c(\|f\|_{L^\infty}\|g\|_{H^r} + \|f\|_{H^r}\|g\|_{L^\infty}),
		\end{equation}
		where $c$ is a constant depending only on $r$.
		\label{product}
	\end{lemma}
	
	\begin{lemma}\cite{kato1988commutator}
		If $r>0$, then
		\begin{equation*}
			\|[\Lambda^r, f]\, g\|_{L^2(\mathbb{R})} \leq c\ (\|\partial_x f\|_{L^\infty(\mathbb{R})} \|\Lambda^{r-1} g\|_{L^2(\mathbb{R})} + \|\Lambda^r f\|_{L^2(\mathbb{R})} \|g\|_{L^\infty(\mathbb{R})}),
		\end{equation*}
		where $c$ is a constant depending only on $r$.
		\label{commutator_estimate}
	\end{lemma}
	% \begin{remark}
		%     Throughout this paper, the integral $\int$ without explicit domain denotes integration over $\mathbb{R}$.
		% \end{remark}
	
	\subsection{Existence of Unique Local Solutions}
	First, we construct the non-local form of this equation by inserting the formulation of $m$ into the equation and applying $(I-\partial_x^2)^{-k}$ to both sides to get
	\begin{align*}
		u_t + (I-\partial_x^2)^{-k}\Big(\big((I-\partial_x^2)^{k} u\big)_x u^p\Big) + b(I-\partial_x^2)^{-k}\Big(\big((I-\partial_x^2)^{k} u\big)u^{p-1}u_x\Big)\\ = (I-\partial_x^2)^{-k}\Big(-(g(u))_x + (b+1)(u^pu_x)\Big).
	\end{align*}
	Let's denote $\Gamma=(I-\partial_x^2)^k$. Then, we obtain
	\begin{equation*}
		u_t + \Gamma^{-1}\big((\Gamma u)_x u^p\big) +  b\Gamma^{-1}\big((\Gamma u)u^{p-1}u_x\big) = -\Gamma^{-1}\big(g(u))_{x} - (b+1)(u^pu_x)\big).
	\end{equation*}
	By using the commutator formulation we have
	\begin{equation*}
		\Gamma^{-1}\big((\Gamma u)_x u^p\big) = \Gamma^{-1}((\Gamma u)_x u^p) = \Gamma^{-1}(\Gamma(u^pu_x) - [\Gamma,u^p]u_x) = u^pu_x - \Gamma^{-1}([\Gamma,u^p]u_x).
	\end{equation*}
	Therefore, the equation under consideration can be rewritten as
	\begin{equation*}
		u_t + u^pu_x = \Gamma^{-1}\Big([\Gamma, u^p]u_x - bu^{p-1}u_x\Gamma u - (g(u))_x + (b+1)u^pu_x \Big),
	\end{equation*}
	or equivalently
	\begin{equation*}
		u_t + A(u)u = f(u)
	\end{equation*}
	where
	\begin{equation}
		A(u)=u^p\partial_x,
	\end{equation}
	and
	\begin{equation}
		f(u)=\Gamma^{-1}\Big([\Gamma, u^p]u_x - bu^{p-1}u_x\Gamma u - (g(u))_x + (b+1)u^pu_x \Big).
		\label{f}
	\end{equation}
	
	To prove this result, we will apply Kato's theorem \ref{kato} with $X=H^{s-1}, Y=H^s$, and $Q=\Lambda=(1-\partial_x^2)^{1/2}$.
	\subsection{Proof of Theorem \ref{thm-local}}
	%\begin{proof}
	\subsubsection*{Proof of Assumption (A1):}
	\begin{lemma}\label{quasi-m-accretive}
		For each $u\in H^s$ with $s>2(k-1)+3/2$, the operator $A(u)=u^p\partial_x$  is quasi-$m$-accretive in $L^2(\R)$.
	\end{lemma}
	\begin{proof}
		Since $L^2$ is a Hilbert space, $A(u)$ is quasi-m-accretive if and only if there is a real number $\beta$ such that
		\begin{itemize}
			\item[(a)] $(A(u)v, v)_{L^2} \geq -\beta\|v\|^2_{L^2}$,
			\item[(b)] The range of $\lambda I+A(u)$ is all of $L^2$ for $\lambda > \beta$.
		\end{itemize}
		For part $(a)$ we have, with $v \in H^1(\mathbb{R})$,
		\begin{equation*}
			|(A(u)v, v)_{L^2}| = |(u^p\partial_xv, v)_{L^2}| = \Bigg|\int u^p\partial_xvv\ dx\Bigg| = \frac 12 \left|\int (u^p)_xv^2\right|
		\end{equation*}
		where $(\cdot,\cdot)_{L^2}$ denotes the inner product in $L^2(\R)$. Since $u \in H^s$ with $s>3/2$, by Sobolev embedding, $u, u_x \in L^{\infty}(\mathbb{R})$ and thus
		\begin{equation*}
			\|(u^p)_x\|_{L^\infty} = \|pu^{p-1}u_x\|_{L^\infty} \leq p\|u\|^{p-1}_{L^{\infty}}\|u_x\|_{L^{\infty}} \leq p\|u\|_{H^s}^p.
		\end{equation*}
		Thus, we get
		\begin{equation*}
			|(A(u)v, v)_{L^2}| \le \frac{1}{2}\|(u^p)_x\|_{L^\infty}\|v\|_{L^2}^2\le \frac{p}{2}\|u\|^p_{H^s}\|v\|^2_{L^2}.
		\end{equation*}
		By setting $\beta=\frac{p}{2}\|u\|^p_{H^s}$, it follows that $(A(u)v, v)_{L^2} \geq -\beta\|v\|_{L^2}^2$.
		
		\medskip
		For part $(b)$, we consider the adjoint operator $A(u)^*$, which is defined by
		\begin{equation*}
			(A(u)v, w)_{L^2} = (v, A(u)^*w)_{L^2},\quad \forall v \in D(A(u)), w\in D(A(u)^*).
		\end{equation*}
		For $v \in H^1(\R), w\in C^\infty_c(\R)$,
		\begin{equation*}
			(A(u)v,w) = \int (u^p\partial_xv)w\ dx = -\int v\partial_x(u^pw)\ dx = (v, -\partial_x(u^pw))_{L^2}.
		\end{equation*}
		Thus, $A(u)^*w= -\partial_x(u^pw)$. The domain $D(A(u)^*)=H^1(\R)$ since
		\begin{equation*}
			\partial_x(u^pw) = (\partial_xu^p)w + u^p\partial_xw,
		\end{equation*}
		and $\partial_xw \in L^2(\R)$, $u^p, \partial_x u^p\in L^\infty(\R)$ due to $w\in H^1(\R)$ and $u\in H^s$ with $s>3/2$. So, $\partial_x(u^pw) \in L^2(\R)$, which implies that $D(A(u))=D(A(u)^*)$.\\~\\
		Also, since $A(u)$ is closed and satisfies (a), the operator $\lambda I + A(u)$ has closed range for $\lambda > \beta$. To show surjectivity, suppose $w \in L^2$ satisfies
		\begin{equation*}
			((\lambda I + A(u)) v, w)_{L^2} = 0 \quad \forall v \in H^1(\mathbb{R}).
		\end{equation*}
		As $H^1(\mathbb{R})$ is dense in $L^2$, this implies $w \in {D}(A(u)^*)$ and $(\lambda I + A(u)^*) w = 0$. 
		For the adjoint operator, we have from part $(a)$
		\begin{equation*}
			(A(u)^*w, w)_{L^2} = (A(u) w, w)_{L^2} \geq -\beta\|w\|^2_{L^2}.
		\end{equation*}
		Thus,
		\begin{equation*}
			0 = ((\lambda I + A(u)^*) w, w)_{L^2} \geq \lambda\|w\|^2_{L^2}-\beta\|w\|^2_{L^2} = ((\lambda - \beta) \|w\|_{L^2}^2.
		\end{equation*}
		For $\lambda > \beta$, this implies $w = 0$. Hence, $\lambda I +A(u)$ is surjective.
	\end{proof}
	\begin{lemma}\cite{pazy2012semigroups}
		Let $\widehat X$ and $\widehat Y$ be two Banach spaces such that $\widehat Y$ is continuously and densely embedded in $\widehat X$. Let $-A$ be the infinitesimal generator of the $C_{0}$-semigroup $T(t)$ on $\widehat X$ and let \(\widehat Q\) be an isomorphism from $\widehat Y$ onto $\widehat X$. Then $\widehat Y$ is $-A$-admissible (i.e. $T (t)\widehat Y \subset \widehat Y$ for all $t \geq 0$, and the restriction of $T(t)$ to $\widehat Y$ is a \(C_{0}\)-semigroup on $\widehat Y$ ) if and only if $-A_{1} = - \widehat QA\widehat Q^{-1}$ is the infinitesimal generator of the $C_{0}$-semigroup $T_{1}(t) = \widehat QT(t)\widehat Q^{-1}$ on $X$. Moreover, if $\widehat Y$ is $-A$-admissible, then the part of $- A$ in $\widehat Y$ is the infinitesimal generator of the restriction $T(t)$ to $\widehat Y$.
		\label{pazy}
	\end{lemma} 
	\begin{lemma}
		For each $u\in H^s$ with $s>2(k-1)+3/2$, the operator $A(u)=u^p\partial_x$  is quasi-$m$-accretive in $H^{s-1}$.
		\label{lemma 2.6}
	\end{lemma}
	\begin{proof}
		Similarly to Lemma \ref{lemma 2.6},  $A(u)=u^p\partial_x$ is quasi-$m$-accretive if and only if there is a real number $\beta$ such that
		\begin{itemize}
			\item[(a)] $(A(u) v, v)_{s-1} \geq -\beta \|v\|_{H^{s-1}}^2$,
			\item[(b)] $-A(u)$ is the infinitesimal generator of a $C_0$-semigroup on $H^{s-1}$.
		\end{itemize}
		First, we will prove part $(a)$. Since $u\in H^s$ with $s>2(k-1)+3/2$, $u$ and $u_x$ belong to $L^\infty$, which follows that $\|u\|_{L^\infty} \leq \|u\|_{s}$ and $\|u_x\|_{L^\infty} \leq \|u\|_{s}$. Then, according to the fact that
		\begin{align*}
			\Lambda^{s-1} (u^p \partial_x v) &= [\Lambda^{s-1}, u^p] \partial_x v + u^p \Lambda^{s-1}(\partial_x v) =[\Lambda^{s-1}, u^p] \partial_x v + u^p \partial_x\Lambda^{s-1} v,
		\end{align*}
		where we recall $\Lambda=(I-\partial_x^2)^{1/2}$. We have
		\begin{align*}
			(A(u) v, v)_{s-1} &= (u^p\partial_xv, v)_{s-1},&\\
			&=(\Lambda^{s-1} (u^p \partial_x v), \Lambda^{s-1} v )_{L^2},&\\
			&=([\Lambda^{s-1}, u^p] \partial_x v, \Lambda^{s-1}v)_{L^2} + (u^p \partial_x\Lambda^{s-1} v, \Lambda^{s-1}v)_{L^2}.
		\end{align*}
		For the commutator term, we apply the commutator estimate Lemma \ref{estimate} with $\sigma = s$, $\tilde s = 0$ and $n = s-1$ to obtain 
		\begin{align*}
			\|[ \Lambda^{s-1}, u^p ] v_x \|_{L^2} &\leq c \|u^p\|_{H^s}\|\partial_x v\|_{H^{s-2}} \leq \tilde{c}\|v\|_{H^{s-1}},
		\end{align*}
		where $\tilde{c}$ is independent of $v$. For the second term, we use the integration by parts to get
		\begin{align*}
			|(u^p \partial_x \Lambda^{s-1} v, \Lambda^{s-1} v )_{L^2}| &= \Big| -\frac{1}{2} \int (u^p)_x (\Lambda^{s-1} v )^2 dx\Big|,&\\
			&\leq c\|(u^p)_x\|_{L^\infty}\|\Lambda^{s-1}v\|_{L^2}^2,&\\
			&\leq \|(u^p)_x\|_{L^\infty}\|v\|^2_{H^{s-1}},&\\
			&\leq \tilde{c} \|v\|^2_{H^{s-1}},
		\end{align*}
		where $\tilde{c}$ is independent of $v$.
		Then, by setting $\beta = \tilde{c}\|u\|_{H^{s}}$ we have get $(A(u) v, v)_{H^{s-1}} \geq -\beta \|v\|_{H^{s-1}}^2$ as required.
		
		\medskip
		Now, we will prove part $(b)$. Let $\widehat Q=\Lambda^{s-1}$ be an isomorphism defined as
		\begin{equation*}
			\widehat Q=\Lambda^{s-1}=(I-\partial_x^2)^{s/2}: H^{s-1} \mapsto L^2.
		\end{equation*}
		Define
		\begin{align*}
			A_1(u):=\widehat QA(u)\widehat Q^{-1} &= \Lambda^{s-1} (u^p\partial_x) \Lambda^{1-s},&\\
			&= u^p\partial_x + [\Lambda, u^p]\partial_x\Lambda^{-1}.
		\end{align*}
		and
		\begin{equation*}
			B_1 := A_1(u)-A(u).
		\end{equation*}
		For $v \in L^2$ and $u\in H^{s-1}$, where $s>2(k-1)+3/2$, we have
		\begin{align*}
			\|B_1(u)\|_{L^2} &= \|[\Lambda^{s-1}, A(u)]\Lambda^{1-s}v\|_{L^2},&\\
			&= \|[\Lambda^{s-1}, u^p]\Lambda^{1-s}\partial_xv\|_{L^2},&\\
			&\leq c \|u^p\|_{H^s}\|\Lambda^{1-s}\partial_xv\|_{H^{s-2}},&\\
			&\leq c(\|u\|_{H^s})|v\|_{L^2},
		\end{align*}
		where we again used the commutator estimate Lemma \ref{estimate} with $\tilde{s}=0, n=s-1$, and $\sigma=s$. So, we obtain that $B_1(u) \in \mathcal{L}(L^2)$.
		
		\medskip
		Recall that from Lemma \ref{quasi-m-accretive} that $-A(u)$ is quasi-$m$-accretive in $L^2$, i.e., $-A(u)$ is the infinitesimal generator of a $C_0$-semigroup on $L^2$. Thus, $A_1(u)=A(u)+B_1(u)$ is also the infinitesimal generator of a $C_0$-semigroup on $L^2$ by the perturbation theorem for semigroup \cite{pazy2012semigroups}. Then, by using the Lemma \ref{pazy} with $\widehat X=L^2, \widehat Y=H^{s-1}$, and $\widehat Q=\Lambda^{s-1}$, we can conclude that $H^{s-1}$ is $-A$-admissible. Therefore, $-A(u)$ is the infinitesimal generator of a $C_0$-semigroup on $H^{s-1}$, which completes the proof of Lemma \ref{lemma 2.6}, and thus assumption $(A1)$.
	\end{proof}
	
	\subsubsection*{Proof of Assumption (A2):}
	
	\begin{lemma}
		Let the operator $A(u)=u^p\partial_x$ with $u \in H^s$, where $s > 2(k - 1) + 3/2$. Then, $A(u) \in \mathcal{L}(H^s, H^{s-1})$ for any $u \in H^s$. Moreover,
		\begin{equation*}
			\|(A(u) - A(v))w\|_{H^{s-1}} \leq \lambda_1 \|u - v\|_{H^{s-1}} \|w\|_{H^s},
		\end{equation*}
		where $u,v,w \in H^s$ and $\lambda_1$ is a constant independent of $w$.
		\label{lemma 2.7}
	\end{lemma}
	\begin{proof}
		For given $u,v,w \in H^s$ with $s>2(k-1)+3/2$, we have
		\begin{align*}
			\|(A(u) - A(v))w\|_{H^{s-1}} = \|(u^p-v^p)\partial_xw\|_{H^{s-1}} &\leq c\|u^p-v^p\|_{H^{s-1}}\|\partial_xw\|_{H^{s-1}},&\\
			&= c \left\|(u - v) \sum_{j=0}^{p-1} u^{p-1-j} v^j\right\|_{H^{s-1}}\|\partial_xw\|_{H^{s-1}},&\\
			&\leq c \|u-v\|_{H^{s-1}} \sum_{j=0}^{p-1} \|u\|_{H^{s-1}}^{p-1-j}\|v\|_{H^{s-1}}^{i}\|w\|_{H^s},&\\
			&\leq c \|u-v\|_{H^{s-1}} \sum_{j=0}^{p-1} \|u\|_{H^s}^{p-1-j}\|v\|_{H^s}^{i} \|w\|_{H^s},&\\
			&= \lambda_1 \|u-v\|_{H^{s-1}}\|w\|_{H^s},
		\end{align*}
		where $\lambda_1$ is a constant. Note that we used the fact that $H^{s-1}$ is an algebra for the estimates. Then, take $v=0$ in the above inequality to obtain $A(u) \in \mathcal{L}(H^s, H^{s-1})$, which completes the proof of Lemma \ref{lemma 2.7}, and thus assumption $(A2)$.
	\end{proof}
	
	\subsubsection*{Proof of Assumption (A3):}
	
	\begin{lemma}
		For any $u \in H^s$, there exists a bounded linear operator $B(u) \in \mathcal{L}(H^{s-1})$ satisfying $B(u)=\Lambda A(u)\Lambda^{-1}-A(u)$, where $B: H^s \mapsto \mathcal{L}(H^{s-1})$ is uniformly bounded on bounded sets in $H^{s}$. Moreover,
		\begin{equation*}
			\|(B(u) - B(v))w\|_{H^{s-1}} \leq \lambda_2 \|u - v\|_{H^s} \|w\|_{H^{s-1}},
		\end{equation*}
		where $u,v \in H^s, w\in H^{s-1}$, and $\lambda_2$ is a constant independent of $w$.
		\label{lemma 2.8}
	\end{lemma}
	\begin{proof}
		Note that
		\begin{equation*}
			B(u) = \Lambda A(u) \Lambda^{-1} - A(u) = \Lambda u^p\partial_x\Lambda^{-1} - u^p\partial_x = [\Lambda, u^p]\Lambda^{-1}\partial_x,
		\end{equation*}
		since $\partial_x$ and $\Lambda$ commute. Therefore, for $w\in H^{s-1}$, we have
		\begin{align*}
			\|B(u)w\|_{H^{s-1}} = \|[\Lambda, u^p]\Lambda^{-1}\partial_xw\|_{H^{s-1}}
			&\leq c \|u^p\|_{H^s}\|\Lambda^{-1}\partial_xw\|_{H^{s-1}},&\\
			&\leq c \|u\|^p_{H^s}\|w\|_{H^{s-1}},&\\
			&\leq \tilde{c} \|w\|_{H^{s-1}},
		\end{align*}
		where $\tilde{c}$ depends on $\|u\|_s$. Moreover,
		\begin{align*}
			\|(B(u) - B(v))w\|_{H^{s-1}} = \|[\Lambda, u^p-v^p]\Lambda^{-1}\partial_xw\|_{H^{s-1}} &\leq c \|u^p - v^p\|_{H^s} \|\Lambda^{-1}\partial_xw\|_{H^{s-1}},&\\
			&\leq c \|u^p - v^p\|_{H^s} \|w\|_{H^{s-1}},&\\
			&\leq \lambda_2 \|u-v\|_{H^s} \|w\|_{H^{s-1}},
		\end{align*}
		where $\lambda_2$ is a constant depending on $\|u\|_s$ and $\|v\|_s$. Here, we again used the Lemma \ref{estimate} with $m=1$ and $\sigma = s$. Take $v=0$ in the above inequality to obtain $B(u) \in \mathcal{L}(H^{s-1})$, which completes the proof of Lemma \ref{lemma 2.8}, and thus the assumption $(A3)$.
	\end{proof}
	
	A major part of this work is concentrating on verifying the local Lipschitz continuity of the nonlinearity $f$ in $H^s$ and $H^{s-1}$, where the latter is more difficult.
	
	\subsubsection*{Proof of Assumption (A4):}
	
	\begin{lemma}
		The nonlinearity $f$ defined in \eqref{f} is uniformly bounded on bounded sets in $H^s$. Moreover, the map $f : H^s \mapsto H^s$ is locally $H^{s-1}$-Lipschitz continuous in the sense that for every $r > 0$ there exists a constant $\lambda_3 > 0$, depending only on $r$, such that
		\begin{equation*}
			\|f(u) - f(v)\|_{H^{s-1}} \leq \lambda_3 \|u - v\|_{H^{s-1}} \quad \text{for all } u, v \in B_r(0) \subseteq H^s
		\end{equation*}
		and locally $H^s$-Lipschitz continuous in the sense that for every $r > 0$ there exists a constant $\lambda_4 > 0$, depending only on $r$, such that
		\begin{equation*}
			\|f(u) - f(v)\|_{H^s} \leq \lambda_4 \|u - v\|_{H^s} \quad \text{for all } u, v \in B_r(0) \subseteq H^s.
		\end{equation*}
		\label{lemma 2.9}
	\end{lemma}
	\begin{proof}
		Recall that
		\begin{equation*}
			f(u) = \Gamma^{-1}\Big([\Gamma, u^p]u_x - bu^{p-1}u_x\Gamma u - (g(u))_x + (b+1)u^pu_x \Big),
		\end{equation*}
		where $\Gamma=(I-\partial_x^2)^k$. Then, we have
		\begin{align*}
			f(u)-f(v) &= \Gamma^{-1}\Big([\Gamma, u^p]u_x-[\Gamma, v^p]v_x - b\left(u^{p-1}u_x\Gamma u - v^{p-1}v_x\Gamma v\right)\\
			&\qquad \quad \quad  - \left(g(u)-g(v)\right)_x + (b+1)(u^pu_x-v^pv_x) \Big)\\
			&= \Gamma^{-1}\Big([\Gamma, u^p]u_x-[\Gamma, v^p]v_x\Big) - b\Gamma^{-1}\Big(u^{p-1}u_x\Gamma u-v^{p-1}v_x\Gamma v\Big)\\
			&\qquad \quad \quad - \Gamma^{-1}\Big(\left(g(u)-g(v)\right)_x\Big) + (b+1)\Gamma^{-1}\Big(u^pu_x-v^pv_x \Big)\\
			& =: \textbf{(F1)} + \textbf{(F2)} + \textbf{(F3)} + \textbf{(F4)}.
		\end{align*}
		
		\medskip
		\underline{\textbf{Proof of the $H^{s-1}$-Lipschitzness}}
		
		\medskip
		\underline{\textit{Estimate} $\textbf{(F1)}$}

		%Then, we can estimate the term as
		%\begin{align*}
		%\|\Gamma^{-1}\big(h_m(u)-h_m(v)\big)\|_{2(k-1)+1/2} &\lesssim \|h_m(u)-h_m(v)\|_{-3/2},&\\
		%&\lesssim \sum_{j=1}^{2k} \big\|\partial_x^j(u^p)\partial_x^{2k-j+1}u-\partial_x^j(v^p)\partial_x^{2k-j+1}v\big\|_{-3/2},&\\
		%&\lesssim \sum_{j=1}^{2k} \|\partial_x^j(u^p)\partial_x^{2k-j+1}(u-v)\|_{-3/2}+\sum_{j=1}^{2k}\|\partial_x^j(u^p-v^p)\partial_x^{2k-j+1}v\|_{-3/2},&\\
		%&= (I) + (II).
		%\end{align*}
		We consider the cases $k = 1$ and $k\ge 2$ separately.
		\begin{itemize}
			\item We first consider the case $k \geq 2$. Expanding $\Gamma=(I-\partial_x^2)^k$ by using the binomial theorem gives
			\begin{equation*}
				\Gamma = \sum_{m=0}^{k} \binom{k}{m} (-1)^m \partial_x^{2m},
			\end{equation*}
			and thus
			\begin{equation*}
				[\Gamma, u^p]u_x = \sum_{m=0}^{k} \binom{k}{m} (-1)^m \left[\partial_x^{2m}(u^p u_x) - u^p \partial_x^{2m} u_x \right].
			\end{equation*}
			By using the product rule, we can write $\partial_x^{2m}(u^p u_x)$ as
			\begin{equation*}
				\partial_x^{2m}(u^p u_x)=\sum_{j=0}^{2m}\binom{2m}{j}\partial_x^j(u^p)\partial_x^{2m-j}(u_x) = \sum_{j=1}^{2m}\binom{2m}{j}\partial_x^j(u^p)\partial_x^{2m-j}(u_x) + u^p\partial_x^{2m}u_x.
			\end{equation*}
			Therefore, the commutator term becomes
			% Define $h_m(u)$ as
			% \begin{equation*}
				%     h_m(u):= \sum_{j=1}^{2m}\binom{2m}{j}\partial_x^j(u^p)\partial_x^{2m-j+1}(u).
				% \end{equation*}
			% Then, the total commutator becomes
			\begin{equation*}
				[\Gamma, u^p]u_x = \sum_{m=0}^{k}\binom{k}{m}(-1)^m\sum_{j=1}^{2m}\binom{2m}{j}\partial_x^j(u^p)\partial_x^{2m-j+1}(u)=: \sum_{m=0}^{k}\binom{k}{m}(-1)^mh_m(u)
			\end{equation*}
			with 
			\begin{equation*}
				h_m(u):= \sum_{j=1}^{2m}\binom{2m}{j}\partial_x^j(u^p)\partial_x^{2m-j+1}(u).
			\end{equation*}
			It follows that
			\begin{equation}\label{ee1}
				\textbf{(F1)} = \sum_{m=0}^{k}\binom{k}{m}(-1)^m\Gamma^{-1}[h_m(u) - h_m(v)].
			\end{equation}
			For each $m\in \{0,1,\ldots, k\}$ we have
			\begin{equation}\label{ee}
				\begin{aligned}
					&\|\Gamma^{-1}\big(h_m(u)-h_m(v)\big)\|_{H^{s-1}}\\
					&\lesssim \|h_m(u)-h_m(v)\|_{H^{s-2k-1}},&\\
					&\lesssim \sum_{j=1}^{2m} \big\|\partial_x^j(u^p)\partial_x^{2m-j+1}u-\partial_x^j(v^p)\partial_x^{2m-j+1}v\big\|_{H^{s-2k-1}}\\
					&\lesssim \sum_{j=1}^{2m} \|\partial_x^j(u^p)\partial_x^{2m-j+1}(u-v)\|_{H^{s-2k-1}} +  \sum_{j=1}^{2m}\|\partial_x^j(u^p-v^p)\partial_x^{2m-j+1}v\|_{H^{s-2k-1}}\\
					&=: (I) + (II).
				\end{aligned}
			\end{equation}
			Of course for $m = 0$, the right hand side vanishes. We first analyze term $(II)$ by considering several cases for the index $j$.
			\begin{itemize}
				\item[\ding{100}] Case $j=1$: We estimate
				\begin{equation*}
					\|\partial_x(u^p-v^p)\partial_x^{2m}v\|_{H^{s-2k-1}} \lesssim \|\partial_x(u^p-v^p)\partial_x^{2m}v\|_{H^{s-2k}}.
				\end{equation*}
				Let $\alpha=s-2k$ and $\beta=s-2$. The conditions $\beta > 1/2$ and $\alpha > -\beta$ are satisfied because $s>2k-1/2>7/2$ (for $k\geq 2$) and $s>k+1$, respectively. Moreover, $\alpha \leq \beta$ holds since $k\geq 1$. Therefore, we can apply the fractional Leibniz rule \ref{leibniz} with $f=\partial_x(u^p-v^p)$ and $g=\partial_x^{2m}v$ to obtain
				\begin{align*}
					\|\partial_x(u^p-v^p)\partial_x^{2m}v\|_{H^{s-2k}} &\lesssim \|\partial_x(u^p-v^p)\|_{H^\beta}\|\partial_x^{2m}v\|_{H^{s-2k}},&\\
					&\lesssim \|u^p-v^p\|_{H^{\beta +1}}\|v\|_{H^{s-2(k-m)}},&\\
					&\lesssim \|u^p-v^p\|_{H^{\beta +1}}\|v\|_{H^{s}}.
				\end{align*}
				This yields
				\begin{align*}
					\|\partial_x(u^p-v^p)\partial_x^{2m}v\|_{H^{s-2k-1}} & \lesssim C\left(\|u\|_{H^s}, \|v\|_{H^s})\right\|u^p-v^p\|_{H^{s-1}}\\
					&\lesssim C\left(\|u\|_{H^s}, \|v\|_{H^s})\right\|u-v\|_{H^{s-1}}.
				\end{align*}
				
				\item[\ding{100}] Case $j=2m$: Let $\alpha=s-2k-1$ and choose $\beta=s-1$. Then $\beta>1/2$ because $s>3/2$. Moreover, $\alpha = s-2k-1 \leq s-1 = \beta$, and the condition $\alpha>-\beta$ becomes $s-2k-1>-(s-1) \iff s>k+1$, which holds because $s>2k-1/2>k+1$ for $k\geq2$. Thus, we can apply the fractional Leibniz rule \ref{leibniz} with $f=\partial_xv$ and $g=\partial_x^{2m}(u^p-v^p)$:
				\begin{align*}
					\|\partial_x^{2m}(u^p-v^p)\partial_xv\|_{H^{s-2k-1}} &\lesssim \|\partial_xv\|_{H^\beta}\|\partial_x^{2m}(u^p-v^p)\|_{H^\alpha},&\\
					&\lesssim \|v\|_{H^{\beta+1}}\|u^p-v^p\|_{H^{s-2(k-m)-1}},&\\
					&\lesssim \|v\|_{H^{\beta+1}}\|u^p-v^p\|_{H^{s-1}}&\\
					&\lesssim C\left(\|u\|_{H^s}, \|v\|_{H^s})\right\|u-v\|_{H^{s-1}}.
				\end{align*}
				\item[\ding{100}] Case $2\leq j \leq 2m-1$: Let $\alpha=s-2k$ and choose $\beta=s-2m+j-1$. Then, $\beta>1/2$ since $s>2k-1/2$ implies $\beta>j-1-1/2\geq1/2$ (since $j\geq2$). The condition  $\alpha \leq \beta$ holds since $s-2k \leq s-2m+j-1 \iff 2m+1 \leq 2k+j$, which is true. For $\alpha > -\beta$, we have $s-2k > -s+2m-j+1 \iff 2s+j > 2(k+m)+1$, which is satisfied since $2s+j > 2(2k-1/2)+j = 4k+j-1.$ Also, since $j\leq2k-1$, we have $s-2k+j \leq s-1$.
				Similarly, apply the fractional Leibniz rule \ref{leibniz} with $f=\partial_x^{2m-j+1}v$ and $g=\partial_x^{j}(u^p-v^p)$:
				\begin{align*}
					\|\partial_x^{j}(u^p-v^p)\partial_x^{2m-j+1}v\|_{H^{s-2k-1}} &\lesssim \|\partial_x^{j}(u^p-v^p)\partial_x^{2m-j+1}v\|_{H^{s-2k}},\\
					&\lesssim \|\partial_x^{j}(u^p-v^p)\|_{H^{s-2k}}\|\partial_x^{2m-j+1}v\|_{H^{s-2m+j-1}},\\
					&\lesssim \|u^p-v^p\|_{H^{s-2k+j}}\|v\|_{H^s}\\
					&\lesssim C(\|u\|_{H^s}, \|v\|_{H^s})\|u-v\|_{H^{s-1}}.
				\end{align*}
			\end{itemize}
			\medskip
			Now, let's go back to the term $(I)$ in \eqref{ee}, where we again consider three cases:
			\begin{itemize}
				\item[\ding{100}] Case $j=1$: By choosing $\alpha=s-2k-1$ and $\beta=s-1$, we can verify that the conditions in Lemma \ref{leibniz} are satisfied, and therefore we can apply it with $f=\partial_x(u^p)$ and $g=\partial_x^{2m}(u-v)$ to obtain
				\begin{align*}
					\|\partial_x(u^p)\partial_x^{2m}(u-v)\|_{H^{s-2k-1}} &\lesssim \|\partial_x(u^p)\|_{H^{s-1}}\|\partial_x^{2m}(u-v)\|_{H^{s-2k-1}},\\
					&\lesssim \|u^p\|_{H^{s}}\|u-v\|_{H^{s-2(k+m)-1}}\\
					&\lesssim C\left(\|u\|_{H^s}, \|v\|_{H^s})\right\|u-v\|_{H^{s-1}}.
				\end{align*}
				\item[\ding{100}] Case $j=2m$: 
				Applying the fractional Leibniz rule in Lemma \ref{leibniz} with $\alpha = s-2k$ and $\beta = s-2$, $f=\partial_x(u-v)$ and $g=\partial_x^{2m}(u^p)$ yields
				\begin{align*}
					\|\partial_x^{2m}(u^p)\partial_x(u-v)\|_{H^{s-2k}} &\lesssim \|\partial_x(u-v)\|_{\beta}\|\partial_x^{2m}(u^p)\|_{H^{s-2k}},\\
					&\lesssim \|u-v\|_{\beta+1}\|u^p\|_{H^{s-2(k-m)}},\\
					&\lesssim C\left(\|u\|_{H^s}, \|v\|_{H^s})\right\|u-v\|_{H^{s-1}}.
				\end{align*}
				\item[\ding{100}] Case $2\leq j \leq 2m-1$: We first estimate
				\begin{equation*}
					\|\partial_x^j(u^p)\partial_x^{2m-j+1}(u-v)\|_{H^{s-2k-1}} \lesssim \|\partial_x^j(u^p)\partial_x^{2m-j+1}(u-v)\|_{H^{s-2k}}.
				\end{equation*}
				Then we apply the fractional Leibniz rule in Lemma \ref{leibniz} with $\alpha = s-2k$, $\beta = s-j-1$, $f=\partial_x^j(u^p)$ and $g=\partial_x^{2m-j+1}(u-v)$ to get
				\begin{align*}
					\|\partial_x^j(u^p)\partial_x^{2m-j+1}(u-v)\|_{H^{s-2k}} &\lesssim \|\partial_x^j(u^p)\|_{\beta}\|\partial_x^{2m-j+1}(u-v)\|_{H^{s-2k}},\\
					&\lesssim \|u^p\|_{\beta+j}\|u-v\|_{H^{s-2(k-m)-j+1}}\\
					&\lesssim C\left(\|u\|_{H^s}, \|v\|_{H^s})\right\|u-v\|_{H^{s-1}}.
				\end{align*}
				% where $\alpha=s-2k$. Choose the same $\beta$ with the previous term, which is $\beta=s-j-1$. Since we consider the same $\alpha$ and $\beta$, all conditions are satisfied. Then, we obtain
				% \begin{equation*}
					%     \|\partial_x^j(u^p)\partial_x^{2k-j+1}(u-v)\|_{H^{s-2k-1}} \lesssim .
					% \end{equation*}
			\end{itemize}
			Combining all these estimates and \eqref{ee1}, we obtain
			\begin{equation*}
				\|\textbf{(F1)}\|_{H^{s-1}} \lesssim C(\|u\|_{H^s}, \|v\|_{H^s})\|u - v\|_{H^{s-1}}
			\end{equation*}
			which is the desired estimate for the case $k\ge 2$.
			
			\item To deal with the case $k=1$, we first use $\Gamma = I - \partial_x^2$ to write%we write the term $(\textbf{F1})$ as
			\begin{align*}
				[\Gamma,u^p]u_x = \partial_x^2(u^pu_x)-u^pu_{xxx} &= \partial_x(pu^{p-1}u_x^2+u^pu_{xx})-u^pu_{xxx}\\
				&= p(p-1)u^{p-2}u_x^3+3pu^{p-1}u_xu_{xx}.
			\end{align*}
			Note that $3pu^{p-1}u_xu_{xx} = \frac{3p}{2}\partial_x(u^{p-1}u_x^2)-\frac{3p(p-1)}{2}u^{p-2}u_x^3$, which gives
			\begin{equation*}
				[\Gamma, u^p]u_x = -\frac{p(p-1)}{2}u^{p-2}u_x^3+\frac{3p}{2}\partial_x(u^{p-1}u_x^2).
			\end{equation*}
			Therefore
			\begin{align*}
				\|\textbf{(F1)}\|_{H^{s-1}} &\lesssim \|\Gamma^{-1}([\Gamma, u^p]u_x - [\Gamma, v^p]v_x)\|_{s-1}\\
				&=\|[\Gamma, u^p]u_x - [\Gamma, v^p]v_x\|_{H^{s-3}}\\
				&\lesssim \underbrace{\|u^{p-2}u_x^3-v^{p-2}v_x^3\|_{H^{s-3}}}_{(A)}+\underbrace{\|\partial_x(u^{p-1}u_x^2)-\partial_x(v^{p-1}v_x^2)\|_{H^{s-3}}}_{(B)}.
			\end{align*}
			By writing
			\begin{equation*}
				u^{p-2}u_x^3-v^{p-2}v_x^3 = (u^{p-2}-v^{p-2})u_x^3 + v^{p-2}(u_x^3-v_x^3),
			\end{equation*}
			we have
			\begin{align*}
				(A) \leq \underbrace{\|(u^{p-2}-v^{p-2})u_x^3\|_{H^{s-3}}}_{(A1)}+\underbrace{\|v^{p-2}(u_x^3-v_x^3)\|_{H^{s-3}}}_{(A2)}.
			\end{align*}
			For $(A1)$, we use the algebra property of $H^{s-1}$ since $s-1>1/2$ under our assumptions to obtain
			\begin{equation*}
				(A1) \lesssim \|u^{p-2}-v^{p-2}\|_{H^{s-1}}\|u_x^3\|_{H^{s-1}} \lesssim \|u-v\|_{H^{s-1}}\|u\|_{H^s}^3.
			\end{equation*}
			For $(A2)$, we factor $u_x^3-v_x^3 = (u-v)(u_x^2+u_xv_x+v_x^2)$ and write
			\begin{equation*}
				(A2) \lesssim \|v^{p-2}(u_x^2+u_xv_x+v_x^2)\partial_x(u-v)\|_{H^{s-2}}.
			\end{equation*}
			Let $\alpha=s-2$ and choose $\beta=s-1$. We see that $\beta > 1/2$ and it satisfies the condition $\alpha > -\beta \iff s-2 > 1-s$, which is okay since $s>2k-1/2$. We now apply the fractional Leibniz rule in Lemma \ref{leibniz} with $f=v^{p-2}(u_x^2+u_xv_x+v_x^2)$ and $g=\partial_x(u-v)$ to have
			\begin{equation*}
				(A2) \lesssim \|v^{p-2}(u_x^2+u_xv_x+v_x^2)\|_{H^{s-1}}\|\partial_x(u-v)\|_{H^{s-2}} \lesssim (\|u\|_{H^s}+\|v\|_{H^s})\|u-v\|_{H^{s-1}}.
			\end{equation*}
			The term $(B)$ is handled similarly. We have
			\begin{equation*}
				(B) = \|\partial_x(u^{p-1}u_x^2-v^{p-1}v_x^2)\|_{H^{s-3}} \lesssim \|u^{p-1}u_x^2-v^{p-1}v_x^2\|_{H^{s-2}}.
			\end{equation*}
			Decomposing the difference as
			\begin{equation*}
				u^{p-1}u_x^2-v^{p-1}v_x^2 = (u^{p-1}-v^{p-1})u_x^2 + v^{p-1}(u_x^2-v_x^2),
			\end{equation*}
			and using the algebra property of $H^{s-1}$ and the fractional Leibniz rule yields
			\begin{equation*}
				(B) \lesssim (\|u\|_{H^s}+\|v\|_{H^s})\|u-v\|_{H^{s-1}}.
			\end{equation*}
			Combining the estimates for $(A)$ and $(B)$, we conclude that
			\begin{equation*}
				\|\textbf{(F1)}\|_{H^{s-1}} \lesssim \|[\Gamma, u^p]u_x - [\Gamma, v^p]v_x\|_{H^{s-3}} \lesssim C(\|u\|_{H^s}, \|v\|_{H^s})\|u-v\|_{H^{s-1}},
			\end{equation*}
			which is the desired estimate for the case $k=1$.
		\end{itemize}
		
		\medskip
		\underline{\textit{Estimate} $\textbf{(F2)}$}
		
		By expanding $\Gamma u = \sum_{m=0}^k \binom{k}{m} (-1)^m \partial_x^{2m} u$ we have
		\begin{equation*}
			- b u^{p-1} u_x \Gamma u = -b \sum_{m=0}^k \binom{k}{m} (-1)^m u^{p-1} u_x \partial_x^{2m} u
		\end{equation*}
		and therefore
		\begin{equation*}
			-b \left( u^{p-1} u_x \Gamma u - v^{p-1} v_x \Gamma v \right) = -b \sum_{m=0}^k \binom{k}{m} (-1)^m \left( u^{p-1} u_x \partial_x^{2m} u - v^{p-1} v_x \partial_x^{2m} v \right).
		\end{equation*}
		It follows that
		\begin{equation}\label{ee2}
			\|\textbf{(F2)}\|_{H^{s-1}} \lesssim \sum_{m=0}^{k}\|\Gamma^{-1}\left( u^{p-1} u_x \partial_x^{2m} u - v^{p-1} v_x \partial_x^{2m} v \right)\|_{H^{s-1}}.
		\end{equation}
		For each $m = 0,\ldots, k$, we have
		\begin{align*}
			&\|\Gamma^{-1}(u^{p-1} u_x \partial_x^{2m} u - v^{p-1} v_x \partial_x^{2m} v)\|_{H^{s-1}}\\
			&\lesssim \|u^{p-1} u_x \partial_x^{2m} u - v^{p-1} v_x \partial_x^{2m} v\|_{H^{s-2k-1}},&\\
			&\lesssim \|\partial_x(u^p)\partial_x^{2m} u - \partial_x(v^p) \partial_x^{2m} v\|_{H^{s-2k-1}},&\\
			&\lesssim \|\partial_x(u^p-v^p)\partial_x^{2m} u\|_{H^{s-2k-1}} + \|\partial_x(v^p) \partial_x^{2m} (u-v)\|_{H^{s-2k-1}},&\\
			&= (III) + (IV).
		\end{align*}
		We consider different cases of $m$.
		\begin{itemize}
			\item Case $m=k$: We estimate
			\begin{equation*}
				(III)=\|\partial_x(u^p-v^p)\partial_x^{2k} u\|_{H^{s-2k-1}} \lesssim \|\partial_x(u^p-v^p)\partial_x^{2k} u\|_{H^{s-2k}}.
			\end{equation*}
			Choose $\alpha=s-2k$ and $\beta=s-2$, we can check as before that they satisfy the conditions in Lemma \ref{leibniz}. Therefore, we can apply this lemma  with $f=\partial_x(u^p-v^p)$ and $g=\partial_x^{2k} u$ to have
			\begin{align*}
				\|\partial_x(u^p-v^p)\partial_x^{2k} u\|_{H^{s-2k}} &\lesssim \|\partial_x(u^p-v^p)\|_{H^{s-2}}\|\partial_x^{2k} u\|_{H^{s-2k}},&\\
				&\lesssim \|u^p-v^p\|_{H^{s-1}}\|u\|_{H^s},&\\
				&\lesssim \|u-v\|_{H^{s-1}}\|u\|_{H^s},
			\end{align*}
			where  gives one of the same cases we considered before. Similarly, we choose $\alpha = s - 2k - 1$ and $\beta = s - 1$ and apply the fractional Leibniz rule in Lemma \ref{leibniz} with $f=\partial_x(v^p)$ and $g=\partial_x^{2k}(u-v)$ to have
			\begin{align*}
				(IV) = \|\partial_x(v^p)\partial_x^{2k}(u-v)\|_{H^{s-2k-1}} &\lesssim \|\partial_x(v^p)\|_{H^{s-1}}\|\partial_x^{2k} (u-v)\|_{H^{s-2k-1}},&\\
				&\lesssim \|v^p\|_{H^s}\|u-v\|_{H^{s-1}},&\\
				&\lesssim C(\|v\|_{H^s})\|u-v\|_{H^{s-1}}.
			\end{align*}
			
			\item Case $m \le  k-1$: We can estimate directly, using $H^{s - 2k + 1}$ is an algebra since $s - 2k + 1 > 1/2$,
			\begin{align*}
				\|\partial_x(u^p - v^p)\partial_x^{2m}u\|_{H^{s-2k-1}} &\lesssim \|\partial_x(u^p - v^p)\partial_x^{2m}u\|_{H^{s-2k+1}}\\
				&\lesssim \|u^p - v^p\|_{H^{s-2k+2}}\|u\|_{H^{s - 2k + 2m + 1}}\\
				&\lesssim \|u^p - v^p\|_{H^{s-1}}\|u\|_{H^{s}}\\
				&\lesssim C(\|u\|_{H^s}, \|v\|_{H^{s}})\|u - v\|_{H^{s-1}}
			\end{align*}
			and
			\begin{align*}
				\|\partial_x(v^p)\partial_x^{2m}(u - v)\|_{H^{s-2k-1}}  &\lesssim \|\partial_x(v^p)\partial_x^{2m}(u - v)\|_{H^{s-2k+1}}\\
				&\lesssim \|v^p\|_{H^{s-2k+2}}\|u-v\|_{H^{s-2k+2m+1}}\\
				&\lesssim C(\|v\|_{L^{\infty}})\|v\|_{H^s}\|u - v\|_{H^{s-1}}\\
				&\lesssim C(\|v\|_{H^s})\|u - v\|_{H^{s-1}}
			\end{align*}
			where we used $s - 2k + 2m + 1 \le s - 1$ for $m\le k-1$.
			% \end{itemize}
	\end{itemize}
	Combining these estimates we obtain
	\begin{equation*}
		\|\textbf{(F2)}\|_{H^{s-1}}\le C(\|u\|_{H^s},\|v\|_{H^s})\|u - v\|_{H^{s-1}}.
	\end{equation*}
	
	\medskip
	\underline{\textit{Estimate} $\textbf{(F3)}$}
	
	We have
	\begin{align*}
		\|\textbf{(F3)}\|_{H^{s-1}} &= \|\Gamma^{-1}\partial_x(g(u) - g(v))\|_{H^{s-1}}\\
		&\lesssim \|g(u) - g(v)\|_{H^{s - 2k}}\\
		&\lesssim \|(u-v)g'(\theta u + (1-\theta)v)\|_{H^{s-1}}\\
		&\lesssim \|g'(\theta u + (1-\theta)v)\|_{H^{s-1}}\|u - v\|_{H^{s-1}}\\
		&\lesssim C(\|u\|_{H^s}, \|v\|_{H^s})\|u - v\|_{H^{s-1}}
	\end{align*}
	thanks to the smoothness of $g$.

	\medskip
	\underline{\textit{Estimate} $\textbf{(F4)}$}
	
	This is done exactly as for $\textbf{(F3)}$ by using $g(u) = (b+1)u^{p+1}/(p+1)$ so we omit it.
	
	\medskip
	\underline{\textbf{Proof of the $H^{s}$-Lipschitzness}}
	
	\medskip
	\underline{\textit{Estimate} $\textbf{(F1)}$}
	
	We use the reformulation \eqref{ee1} and estimate
	\begin{align*}
		&\|\Gamma^{-1}\big(h_n(u)-h_m(v)\big)\|_{H^s}\\
		&\lesssim \|h_m(u)-h_m(v)\|_{H^{s-2k}},&\\
		&\lesssim \sum_{j=1}^{2m} \big\|\partial_x^j(u^p)\partial_x^{2m-j+1}u-\partial_x^j(v^p)\partial_x^{2m-j+1}v\big\|_{H^{s-2k}},&\\
		&\lesssim \sum_{j=1}^{2m} \|\partial_x^j(u^p)\partial_x^{2m-j+1}(u-v)\|_{H^{s-2k}}+\sum_{j=1}^{2m}\|\partial_x^j(u^p-v^p)\partial_x^{2m-j+1}v\|_{H^{s-2k}},&\\
		&= (V) + (VI).
	\end{align*}
	
	We will follow the similar ideas with the proof of the $H^{s-1}$-Lipschitzness. Concernig the sum $(VI)$, since in the cases for $j=1$ and $2\leq j\leq 2k-1$, we already estimated $(s-2k)$-norm, we obtain immediately
	\begin{align*}
		\|\partial_x^j(u^p)\partial_x^{2m-j+1}(u-v)\|_{H^{s-2k}} &\lesssim C(\|u\|_{H^s},\|v\|_{H^s})\|u - v\|_{H^{s-1}}\\
		&\lesssim C(\|u\|_{H^s},\|v\|_{H^s})\|u - v\|_{H^{s}}
	\end{align*}
	and
	\begin{align*}
		\|\partial_x^j(u^p-v^p)\partial_x^{2m-j+1}v\|_{H^{s-2k}} &\lesssim C(\|u\|_{H^s},\|v\|_{H^s})\|u - v\|_{H^{s-1}}\\
		&\lesssim C(\|u\|_{H^s},\|v\|_{H^s})\|u - v\|_{H^{s}}.
	\end{align*}
	It remains to look at the case $j=2m$. We choose $\alpha = s - 2k$ and $\beta = s-1$. It can be checked that $\beta > -\alpha$ and they satisfy the conditions in Lemma \ref{leibniz}. Therefore, by applying this lemma with $f = \partial_x$ and $g= \partial_x^{2m}(u^p - v^p)$ it yields
	\begin{align*}
		\|\partial_x^{2m}(u^p-v^p)\partial_xv\|_{H^{s-2k}} &\lesssim \|\partial_x^{2m}(u^p-v^p)\|_{H^{s-2k}}\|\partial_x v\|_{H^{s-1}},&\\
		&\lesssim \|u^p-v^p\|_{H^{s-2k+2m}}\|v\|_{H^s},&\\
		&\lesssim \|u^p-v^p\|_{H^s}\|v\|_{H^s}&\\
		&\lesssim C(\|u\|_{H^s}, \|v\|_{H^s})\|u - v\|_{H^s}&.
	\end{align*}
	For the terms in the sum $(V)$, the cases $j=2m$ and $2\leq j\leq2m-1$ can be inferred from the proof the $H^{s-1}$-Lipschitzness. For $j = 1$ we choose $\alpha = s-2k$ and $s - 1$, then apply Lemma \ref{leibniz} with $f = \partial_xu^p$ and $g = \partial_x^{2m}(u-v)$ to get
	\begin{align*}
		\|\partial_x(u^p)\partial_x^{2m}(u-v)\|_{H^{s-2k}} &\lesssim \|\partial_x u^p\|_{H^{s-1}}\|\partial_x^{2m}(u-v)\|_{H^{s-2k}},&\\
		&\lesssim \|u^p\|_{H^s}\|u-v\|_{H^{s-2k+2m}}&\\
		&\lesssim C(\|u\|_{H^s})\|u - v\|_{H^s}.
	\end{align*}
	Combining all these, we have
	\begin{equation*}
		\|\textbf{(F1)}\|_{H^s} \lesssim C(\|u\|_{H^s},\|v\|_{H^s})\|u - v\|_{H^s}.
	\end{equation*}

	\medskip
	\underline{\textit{Estimate} $\textbf{(F2)}$}
	
	Similar to \eqref{ee2} we have
	\begin{align*}
		\|\textbf{(F2)}\|_{H^s} &\lesssim \sum_{m=0}^{k}\|\Gamma^{-1}\left( u^{p-1} u_x \partial_x^{2m} u - v^{p-1} v_x \partial_x^{2m} v \right)\|_{H^{s}} \\
		&\lesssim \sum_{m=0}^{k}\|u^{p-1} u_x \partial_x^{2m} u - v^{p-1} v_x \partial_x^{2m} v\|_{H^{s-2k}},&\\
		&\lesssim \sum_{m=0}^{k}\|\partial_x(u^p)\partial_x^{2m} u - \partial_x(v^p) \partial_x^{2m} v\|_{H^{s-2k}},&\\
		&\lesssim \sum_{m=0}^{k}\|\partial_x(u^p-v^p)\partial_x^{2m} u\|_{H^{s-2k}} + \|\partial_x(v^p) \partial_x^{2m} (u-v)\|_{H^{s-2k}},&\\
		&= (VI) + (VII).
	\end{align*}
	For $(VI)$ we use the fact that $H^{s-2k}$ and $H^{s-2k+1}$ are algebras to estimate
	\begin{align*}
		(VI) &\lesssim \sum_{m=0}^k\|u^p - v^p\|_{H^{s-2k+1}}\|u\|_{H^{s-2k+2m}}\\
		&\lesssim \sum_{m=0}^kC(\|u\|_{H^{s-2k+1}}, \|v\|_{H^{s-2k+1}})\|u-v\|_{H^{s-2k+1}}\|u\|_{H^s}\\
		&\lesssim C(\|u\|_{H^s},\|v\|_{H^s})\|u - v\|_{H^s}.
	\end{align*}
	Concerning $(VII)$ we also use the algebraic structure of $H^{s-2k}$ to have
	\begin{align*}
		(VII) &\lesssim  \sum_{m=0}^k\|\partial_x(v^p)\|_{H^{s-2k}}\|\partial_x^{2m} (u-v)\|_{H^{s-2k}}\\
		&\lesssim \|v^p\|_{H^{s-2k+1}}\|u-v\|_{H^{s-2k+2m}},\\
		&\lesssim C(\|v\|_{L^\infty})\|v\|_{H^{s-2k+1}}\|u-v\|_{H^{s}},\\
		&\lesssim C(\|u\|_{H^s},\|v\|_{H^s})\|u-v\|_{H^s}.
	\end{align*}
	Therefore,
	\begin{equation*}
		\|\textbf{(F2)}\|_{H^s} \le C(\|u\|_{H^s},\|v\|_{H^s})\|u - v\|_{H^s}.
	\end{equation*}
	
	\medskip
	\underline{\textit{Estimate} $\textbf{(F3)}$}
	
	Direct computations give
	\begin{align*}
		\|\textbf{(F3)}\|_{H^s} &= \|\Gamma^{-1}\left((g(u))_x-(g(v))_x\right)\|_{s}\\
		&\lesssim \|g(u) - g(v) \|_{H^{s-2k+1}}\\
		&\lesssim \|(u-v)g'(\theta u + (1-\theta) v)\|_{H^{s}}\\
		&\lesssim \|g'(\theta u + (1-\theta)v)\|_{H^s}\|u - v\|_{H^s}\\
		&\lesssim C(\|u\|_{H^s}, \|v\|_{H^s})\|u - v\|_{H^s}.
	\end{align*}
	
	\medskip
	\underline{\textit{Estimate} $\textbf{(F4)}$}
	
	Again, this follows from the estimate of $\textbf{(F3)}$ with $g(u) = (b+1)u^{p+1}/(p+1)$ so we omit it.
	% For the term $(b+1)(u^pu_x-v^pv_x)$, we have
	% \begin{equation*}
		%     \|\Gamma^{-1}(b+1)(u^pu_x-v^pv_x)\|_{H^s} \lesssim \|u^pu_x-v^pv_x\|_{H^{s-2k}}.
		% \end{equation*}
	% We estimate the term by splitting
	% \begin{equation*}
		%     u^pu_x-v^pv_x = u^p(u_x-v_x)+(u^p-v^p)v_x.
		% \end{equation*}
	% Then, we have
	% \begin{align*}
		%     \|u^pu_x-v^pv_x\|_{H^{s-2k}} &\lesssim \|u^p(u_x-v_x)+(u^p-v^p)v_x\|_{H^{s-2k}},&\\
		%     &\lesssim \|u^p(u_x-v_x)\|_{H^{s-2k}} + \|(u^p-v^p)v_x\|_{H^{s-2k}}.
		% \end{align*}
	% Choose $\beta = s-1$. The condition $\beta > |\alpha|$ is equivalent to $s-1>|s-2k|$, which holds since $s>2k-1/2$ and consequently $s>k+1/2$ for all $k\geq1$. Then, applying the fractional Leibniz rule (Lemma \ref{leibniz}) with this choice yields
	% \begin{equation*}
		%     \|u^p(u_x-v_x)\|_{H^{s-2k}} \lesssim \|u^p\|_{H^{s-1}}\|u_x-v_x\|_{H^{s-2k}} \lesssim \|u\|^p_{H^{s}}\|u-v\|_{H^{s-2k+1}} \lesssim \|u\|_{H^{s}}\|u-v\|_{H^{s}},
		% \end{equation*}
	% since $s-2k+1 \leq s$. Similarly,
	% \begin{equation*}
		%     \|(u^p-v^p)v_x\|_{H^{s-2k}} \lesssim \|u^p-v^p\|_{H^{s-1}}\|v_x\|_{H^{s-2k}} \lesssim \|u-v\|_{H^s}\|v\|_{H^{s-2k+1}} \lesssim \|u-v\|_{H^s}\|v\|_{H^{s}}.
		% \end{equation*}
	% Combining these two estimates gives
	% \begin{equation*}
		%      \|u^pu_x-v^pv_x\|_{H^{s-2k}} \lesssim C(\|u\|_{H^s}, \|v\|_{H^s}) \|u-v\|_{H^s}.
		% \end{equation*}
\end{proof}

Up to this point, we have verified all assumptions $(A1)-(A4)$ in Theorem \ref{Kato} with $X = H^{s-1}$ and $Y = H^s$. This implies the local existence of a unique solution to \eqref{gch} and thus completes the proof of Theorem \ref{thm-local}.
%\end{proof}
\section{Global Existence}\label{global}
\subsection{Conservation Laws and Uniform Bounds}
\begin{lemma}
	Let $u$ be the local solution obtained in Theorem \ref{thm-local} in the maximal interval $(0,T_{\max})$ and assume that $b = p+1$. Then we have the following conversation law
	\begin{align*}
		I_1(t) := \int u(t)m(t)\ dx = \int u_0m_0dx \quad \text{ for all } t\in (0,T_{\max}).
	\end{align*}
\end{lemma}
\begin{proof}
	Compute
	\begin{equation*}
		\frac{d}{dt} \int I_1\ dx = \int (u_tm + um_t)\ dx.
	\end{equation*}
	From the definition of $m = (I-\partial_x^2)^ku$ and integration by parts we have
	\begin{equation*}
		\int u_tm\ dx = \int u_t (I-\partial_x^2)^ku\ dx = \int (I-\partial_x^2)^ku_tu\ dx = \int m_tu\ dx.
	\end{equation*}
	Thus,
	\begin{align*}
		\frac{d}{dt}I_1 &= 2\int um_t\ dx\\
		&=  2\int u(- m_xu^p - bmu^{p-1}u_x -(g(u))_x + (b+1)u^pu_x)\ dx\\
		&= - 2\int u^{p+1}m_x\ dx - 2b\int u^pu_xm\ dx - 2\int u(g(u))_x\ dx + 2(b+1)\int u^{p+1}u_x\ dx.
	\end{align*}
	For the terms on the right hand side, we evaluate them as
	\begin{equation*}
		\int u^{p+1}m_x\ dx = - \int (u^{p+1})_xm\ dx = - \int (p+1)u^pu_xm\ dx,
	\end{equation*}
	\begin{equation*}
		-\int u(g(u))_x\ dx = \int \frac{d}{dx}G(u)\ dx = 0,
	\end{equation*}
	where $G(u)$ is an antiderivative of $g$, and
	\begin{equation*}
		(b+1)\int u^{p+1}u_x\ dx = \frac{b+1}{p+1} \int (u^{p+1})_x\ dx = 0.
	\end{equation*}
	Collecting these terms, we have
	\begin{equation*}
		\frac{d}{dt}I_1 = 2(p+1-b)\int u^pu_xm\ dx = 0
	\end{equation*}
	due to the assumption $b = p+1$.
\end{proof}
\begin{remark}
	The paper \cite{mclachlan2009well} considered $p=1$ and $b=2$, which is a special case of our results.
\end{remark}
\begin{lemma}
	Under the same assumptions as in Theorem \ref{global-wellposedness}, there exists a constant $M>0$ such that for $t\geq0$,
	\begin{equation*}
		\|u(t)\|_{H^k} \leq M.
	\end{equation*}
	Moreover, by Sobolev embedding, we have
	\begin{equation*}
		\|u(t)\|_{L^\infty} \leq CM, \quad \quad\|u_x(t)\|_{L^\infty} \leq CM,
	\end{equation*}
	where $C$ is a constant depending on $k$.
\end{lemma}

\begin{proof}
	The conservation of $I_1$ implies that $I_1(t) = I_1(0)$ for all $t$. Since $m=(I-\partial_x^2)^k$, we have
	\begin{equation*}
		I_1 = \int um\ dx = \int u(I-\partial_x^2)^ku\ dx
	\end{equation*}
	By  using the binomial theorem
	we obtain
	\begin{equation*}
		I_1 = \sum_{j=0}^k (-1)^j \binom{k}{j} \int u\partial_x^{2j}u\ dx.
	\end{equation*}
	After integration by parts $j$ times, we get
	\begin{equation*}
		I_1 = \sum_{j=0}^k \binom{k}{j} \int (\partial_x^ju)^2\ dx,
	\end{equation*}
	which is equivalent to the $H^k$-norm squared, i.e, there exists constants $c, C > 0$ such that
	\begin{equation*}
		c\|u\|^2_{H^k} \leq I_1 \leq C\|u\|^2_{H^k},
	\end{equation*}
	where $\|u\|^2_{H^k} = \sum_{j=0}^{k}\int (\partial_x^ju)^2\ dx$. Thus, conservation of $I_1$ implies a uniform bound
	\begin{equation*}
		\|u(t)\|_{H^k} \leq M.
	\end{equation*}
	By the Sobolev embedding theorem, for $k\geq2$ we have $H^k(\mathbb{R}) \hookrightarrow C^1(\mathbb{R})$, which implies that there exists a constant $C_S>0$ such that for any $u \in H^k(\mathbb{R})$
	\begin{equation*}
		\|u(t)\|_{L^\infty} \leq C_S \|u(t)\|_{H^k}, \quad \|u_x(t)\|_{L^\infty} \leq C_S \|u(t)\|_{H^k}.
	\end{equation*}
	Therefore, unifrom bound on $\|u(t)\|_{H^k}$ yields uniform bounds
	\begin{equation*}
		\|u(t)\|_{L^\infty} \leq C_S M, \quad \|u_x(t)\|_{L^\infty} \leq C_S M \quad \text{for all } t \geq 0.
	\end{equation*}
	These uniform controls over $u$ and $u_x$ in $L^\infty(\R)$, which are crucial for controlling nonlinear terms in the energy estimates later.
\end{proof}
\begin{proposition}
	Let $u$ be the local solution of \eqref{gch} obtained in Theorem \ref{thm-local}. If the initial data $m_0 \in L^2(\R)$, then $m(t) \in  L^2(\R)$ for all $t\in [0,T_{\max})$. Moreover, we have the estimate
	\begin{equation*}
		\|m\|_{L^2} \leq e^{C_0t}\|m_0\|_{L^2},
	\end{equation*}
	where $C_0$ is a constant depending only on the norm of the initial value $u_0$.
\end{proposition}
\begin{proof}
	% With $b = p+1$, the quantity $I_1(t) = \int u m  dx$ is conserved, so
	%     \begin{equation*}
		%         I_1(t) = I_1(0) \quad \text{for all } t \geq 0.
		%     \end{equation*}
	% Since $m = (I - \partial_x^2)^k u$, we have
	%     \begin{equation*}
		%         I_1 = \int u m\ dx = \int u (I - \partial_x^2)^k u\ dx.
		%     \end{equation*}
	% This is equivalent to the $H^k$-norm squared of $u$, so there exists constants $c,C >0$ such that
	%     \begin{equation*}
		%         c\|u\|_{H^k}^2 \leq I_1 \leq C\|u\|_{H^k}^2.
		%     \end{equation*}
	% Thus,
	%     \begin{equation*}
		%         \|u(t)\|_{H^k} \leq M, \quad \text{where } M = \sqrt{\frac{I_1(0)}{c}}.
		%     \end{equation*}
	% Since $k \geq 2$, Sobolev embedding gives
	%     \begin{equation*}
		%          \|u(t)\|_{L^\infty} \leq C_S\|u\|_{H^k} \leq C_SM, \quad \|u_x(t)\|_{L^\infty} \leq C_S\|u\|_{H^k} \leq C_SM,
		%     \end{equation*}
	% where $C_S$ is the Sobolev constant. Set $C_M=C_SM$, so that
	%     \begin{equation*}
		%         \|u(t)\|_{L^\infty} \leq C_M, \quad \|u_x(t)\|_{L^\infty} \leq C_M.
		%     \end{equation*}    
	%(\textcolor{red}{\textbf{UP TO HERE!}}) 
	Multiply equation (\ref{gch}) by $m$ and integrate to get, noting that $b = p+1$,
	\begin{equation*}
		\int m m_t  dx + \int m m_x u^p  dx + (p+1) \int m^2 u^{p-1} u_x  dx = - \int m (g(u))_x  dx + (p+2) \int m u^p u_x  dx.
	\end{equation*}
	By using integration by parts
	\begin{align*}
		\int m m_x u^p  dx &= -\frac{p}{2} \int u^{p-1} u_x m^2  dx,
	\end{align*}
	and it leads to
	\begin{equation}
		\frac{1}{2} \frac{d}{dt} \|m\|_{L^2}^2 + (\frac{p}{2} + 1) \int u^{p-1} u_x m^2  dx = - \int m (g(u))_x\ dx + (p+2) \int m u^p u_x\ dx.
		\label{eqn_estimate}
	\end{equation}
	For the term involving $g(u)$, it holds
	\begin{equation*}
		\left|- \int m (g(u))_x  dx\right| = \left|\int m g'(u) u_x dx\right| \leq \|m\|_{L^2} \|g'(u) u_x\|_{L^2}.
	\end{equation*}
	Since $\|u\|_{L^\infty} \leq C_M$, we can define $G_M = \sup_{|z| \leq C_M} |g'(z)| < +\infty$. Also, $\|u_x\|_{L^2} \leq \|u\|_{H^1} \leq \|u\|_{H^k} \leq M$. Hence,
	\begin{equation*}
		\|g'(u) u_x\|_{L^2} \leq G_MM,
	\end{equation*}
	and so,
	\begin{equation*}
		\left|- \int m (g(u))_x  dx\right| \leq G_MM \|m\|_{L^2}.
	\end{equation*}
	Using the bounds of $\|u\|_{L^\infty}$ and $\|u_x\|_{L^\infty}$, we can also estimate
	\begin{equation*}
		\left|\left(\frac{p}{2} + 1\Big) \int u^{p-1} u_x m^2  dx \right| \leq \left(\frac{p}{2} + 1\right) \|u^{p-1}\|_{L^\infty}\|u_x\|_{L^\infty} \|m\|_{L^2}^2 \leq \Big(\frac{p}{2} + 1 \right) C_M^p \|m\|_{L^2}^2,
	\end{equation*}
	\begin{equation*}
		\left|(p+2) \int m u^p u_x\ dx\right| \leq (p+2) \|u^p\|_{L^\infty} \|u_x\|_{L^\infty} \|m\|_{L^2} \leq (p+2) C_M^{p+1} \|m\|_{L^2}.
	\end{equation*}
	Substituting into \eqref{eqn_estimate} yields
	\begin{equation*}
		\frac{1}{2} \frac{d}{dt} \|m\|_{L^2}^2 \leq \Big(\frac{p}{2} + 1\Big) C_M^{p} \|m\|_{L^2}^2 + G_MM\|m\|_{L^2} + (p+2) C_M^{p+1} \|m\|_{L^2}.
	\end{equation*}
	% Multiplying by 2 gives
	%     \begin{equation*}
		%         \frac{d}{dt} \|m\|_{L^2}^2 \leq (p+2) C_M^{p} \|m\|_{L^2}^2 + 2(p+2) C_M^{p+1} \|m\|_{L^2} + 2C G_M\|m\|_{L^2}.
		%     \end{equation*}
	% Let
	or
	\begin{equation*}
		\frac{d}{dt} \|m\|_{L^2}^2 \leq K_1 \|m\|_{L^2}^2 + K_2 \|m\|_{L^2} \le K \|m\|_{L^2}^2 + L
	\end{equation*}
	with
	\begin{align*}
		K_1 = (p+2) C_M^p, \quad 
		K_2 = 2(p+2) C_M^{p+1} + 2CG_M, \quad K = K_1 + 1, \quad L = \frac{K_2^2}{4}.
	\end{align*}
	% Use Young's inequality to get
	%     \begin{equation*}
		%         K_2 \|m\|_{L^2} \leq \frac{K_2^2}{4} + \|m\|_{L^2}^2.
		%     \end{equation*}
	% So, we have
	%     \begin{equation*}
		%         \frac{d}{dt} \|m\|_{L^2}^2 \leq (K_1 + 1) \|m\|_{L^2}^2 + \frac{K_2^2}{4}.
		%     \end{equation*}
	% Let $K = K_1 + 1$ and $L = \frac{K_2^2}{4}$. Then,
	% \begin{equation*}
		%     \frac{d}{dt} \|m\|_{L^2}^2 \leq K \|m\|_{L^2}^2 + L.
		% \end{equation*}
	By Gronwall's inequality, we get
	\begin{equation*}
		\|m(t)\|_{L^2}^2 \leq \Big( \|m_0\|_{L^2}^2 + \frac{L}{K} \Big) e^{K t} - \frac{L}{K}
	\end{equation*}
	which finishes the proof of the lemma with $C_0 = K/2$. 
\end{proof}
\begin{corollary}
	Assume $b=p+1$, $m_0 \in L^2(\R)$, and $u_0 \in H^s(\R)$ for $s>2k-\frac{1}{2}$. Then, we have the following estimate
	\begin{equation*}
		\|u(t)\|_{H^{2k}} \leq C(t),\quad \text{for all}\ t \in (0,T_{max})
	\end{equation*}
	where $C(t)$ depends continuously in $t\in (0,\infty)$.
\end{corollary}
\begin{proof}
	Let $N=\|\partial_x^{2k}u(t)\|_{L^2}$. Using $m=(I-\partial_x^2)^ku$, we write
	\begin{equation*}
		\partial_x^{2k}u = (-1)^k\left[m-\sum_{j=0}^{k-1}(-1)^j\binom{k}{j}\partial_x^{2j}u \right]
	\end{equation*}
	Taking the $L^2$-norm gives
	\begin{equation}
		N \leq \|m\|_{L^2} + \sum_{j=0}^{k-1}\binom{k}{j}\|\partial_x^{2j}u\|_{L^2}.
		\label{eqn_L^2norm}
	\end{equation}
	We estimate the terms $\|\partial_x^{2j}u\|_{L^2}$ for $j=0, ..., k-1$.
	\begin{itemize}
		\item If $2j \leq k$, then by the uniform $H^k$ bound,
		\begin{equation*}
			\|\partial_x^{2j}u\|_{L^2} \leq \|u\|_{H^k} \leq M.
		\end{equation*}
		\item If $2j > k$, we use the Gagliardo–Nirenberg interpolation inequality
		\begin{equation*}
			\|\partial_x^{2j}u\|_{L^2} \leq C\|\partial_x^{2k}u\|^{\alpha_j}_{L^2}\|u\|^{1-\alpha_j}_{H^k},
		\end{equation*}
		where $\alpha_j = \frac{2j-k}{k} \in (0,1)$. Since $\|u\|_{H^k} \leq M$, we obtain
		\begin{equation*}
			\|\partial_x^{2j}u\|_{L^2} \leq CM^{1-\alpha_j}N^{\alpha_j}.
		\end{equation*}
		Let's define the index sets:
		\begin{equation*}
			J_1={j: 0 \leq j \leq k-1, 2j \leq k}, \quad J_2={j: 0 \leq j \leq k-1, 2j > k}.
		\end{equation*}
		Then, from \eqref{eqn_L^2norm},
		\begin{equation*}
			N \leq \|m\|_{L^2} + \sum_{j\in J_1}\binom{k}{j}M + \sum_{j \in J_2}\binom{k}{j}CM^{1-\alpha_j}N^{\alpha_j}.
			\label{eqn_estimate_N}
		\end{equation*}
		Let
		\begin{equation*}
			C_1 = \sum_{j\in J_1}\binom{k}{j}M, \quad C_2=\sum_{j \in J_2}\binom{k}{j}CM^{1-\alpha_j}
		\end{equation*}
		Then, \eqref{eqn_estimate_N} becomes
		\begin{equation*}
			N \leq \|m\|_{L^2} + C_1 + C_2\sum_{j \in J_2}N^{\alpha_j}.
		\end{equation*}
		For each $j \in J_2$, apply Young's inequality with $\epsilon_j >0$
		\begin{equation*}
			N^{\alpha_j} \leq \epsilon_jN + C(\epsilon_j).
		\end{equation*}
		Select $\epsilon_j$ so that $C_2\sum_{j \in J_2}\epsilon_j < \frac{1}{2}$. Then,
		\begin{equation*}
			N \leq \|m\|_{L^2} + C_1 + C_2\left(\left(\sum_{j \in J_2}\epsilon_j\right)N + \sum_{j \in J_2}C(\epsilon_j)\right)
		\end{equation*}
		Rearranging as
		\begin{equation*}
			\left(1-C_2\sum_{j \in J_2}\epsilon_j\right)N \leq \|m\|_{L^2} + C_1 + C_2\sum_{j \in J_2}C(\epsilon_j).
		\end{equation*}
		Since $1-C_2\sum_{j \in J_2}\epsilon_j > \frac{1}{2}$, we obtain
		\begin{equation*}
			N \leq 2\left(\|m\|_{L^2} + C_1 + C_2\sum_{j \in J_2}C(\epsilon_j)\right).
		\end{equation*}
		Finally, using $\|m\|_{L^2} \leq e^{C_0t}\|m_0\|_{L^2}$, we conclude
		\begin{equation*}
			\|u(t)\|_{H^{2k}} \leq C\left(\|u(t)\|_{H^{k}} + N\right) \leq C\left(M + e^{C_0t}\|m_0\|_{L^2} + C_1 + C_2\sum_{j \in J_2}C(\epsilon_j)\right) = C(t),
		\end{equation*}
		which completes the proof.
	\end{itemize}
\end{proof}
\subsection{Proof of Theorem \ref{global-wellposedness}}
\begin{proof}
	We use energy estimates to show that the $H^s$-norm stays finite as $t\to T_{\max}$, implying the global existence of the local solution obtained in Theorem \ref{thm-local}.
	% Differentiate $\|u\|_{H^s}^2 = \int (I-\partial_x^2)^{s/2}u \cdot (I-\partial_x^2)^{s/2}u\ dx$ with respect to $t$:
	% \begin{equation*}
		%     \frac{d}{dt}\|u\|_{H^s}^2 = 2\big((I-\partial_x^2)^{s/2}u, (I-\partial_x^2)^{s/2}u_t\big)_{L^2}.
		% \end{equation*}
	% Let $\Lambda = (I-\partial_x^2)^{1/2}$, which implies that $\Lambda^s=(I-\partial_x^2)^{s/2}$. Then,
	% \begin{equation*}
		%     \frac{d}{dt}\|u(t)\|_{H^s}^2 = 2(\Lambda^su, \Lambda^su_t)_{L^2}.
		% \end{equation*}
	Recall the operator $\Lambda = (I - \partial_x^2)^{1/2}$. We use the equivalent $H^s$-norm $\|u\|_{H^s}^2 = \|\Lambda^s u\|_{L^2}^2 = (\Lambda^s u, \Lambda^s u)_{L^2}$ and  the quasilinear form  $u_t + A(u)u = f(u)$ to get
	\begin{equation*}
		\frac{d}{dt}\|u(t)\|_{H^s}^2 = -2\underbrace{(\Lambda^su, \Lambda^s(A(u)u))_{L^2}}_{(*)} + 2\underbrace{(\Lambda^su, \Lambda^sf(u))_{L^2}}_{(**)}.
	\end{equation*}
	Write $\Lambda^s(A(u)u)$ as
	\begin{equation*}
		\Lambda^s(A(u)u) = \Lambda^s(u^pu_x) = u^p\Lambda^su_x + [\Lambda^s, u^p]u_x,
	\end{equation*}
	and thus
	\begin{equation*}
		(*) = (\Lambda^su, \Lambda^s(A(u)u))_{L^2} = \underbrace{(\Lambda^su, u^p\Lambda^su_x)_{L^2}}_{(*_1)} + \underbrace{(\Lambda^su, [\Lambda^s, u^p]u_x)_{L^2}}_{(*_2)}.
	\end{equation*}
	For $(*_1)$, since the operators $\Lambda^s$ and $\partial_x$ commute, we have
	\begin{equation*}
		(\Lambda^su, u^p\Lambda^su_x)_{L^2} = \int \Lambda^suu^p\partial_x(\Lambda^su)\ dx = \frac{1}{2}\int u^p\partial_x(\Lambda^su)^2\ dx = -\frac{1}{2}\int \partial_x(u^p)(\Lambda^su)^2\ dx.
	\end{equation*}
	Substituting $\partial_x(u^p)=pu^{p-1}u_x$ yields
	\begin{equation*}
		|(\Lambda^su, u^p\Lambda^su_x)_{L^2}| \leq \frac{p}{2}\|u^{p-1}u_x\|_{L^\infty}\|\Lambda^su\|^2_{L^2} \leq C\|u\|_{H^s}^2
	\end{equation*}
	thanks to the boundedness of $\|u\|_{L^\infty}$ and $\|u_x\|_{L^\infty}$. For $(*_2)$, applying Lemma \ref{commutator_estimate} gives
	\begin{align*}
		\|[\Lambda^s, u^p]u_x\|_{L^2} &\lesssim C\big(\|\partial_x(u^p)\|_{L^\infty}\|\Lambda^{s-1}u_x\|_{L^2} + \|\Lambda^s(u^p)\|_{L^2}\|u_x\|_{L^\infty}\big)\\
		&\lesssim \|u\|_{L^\infty}^{p-1}\|u_x\|_{L^\infty}\|u\|_{H^s} + \|u^p\|_{H^s}\|u_x\|_{L^\infty}\\
		&\lesssim C(\|u\|_{L^\infty}, \|u_x\|_{L^\infty})\|u\|_{H^s}
	\end{align*}
	where we used $\|\Lambda^s(u^p)\|_{L^2} \leq C(\|u\|_{L^\infty})\|u\|_{H^s}$ thanks to Lemma \ref{est1}, and the bounds of $\|u\|_{L^\infty}$ and $\|u_x\|_{L^\infty}$ at the last step. Therefore,
	\begin{equation*}
		|(*_2)| = |(\Lambda^su, [\Lambda^s, u^p]u_x)_{L^2}| \leq \|\Lambda^su\|_{L^2}\|[\Lambda^s, u^p]u_x\|_{L^2} \leq \|u\|_{H^s}\|u\|_{H^s} = C\|u\|_{H^s}^2.
	\end{equation*}
	Thus we get the desired estimate $|(*)| \leq C\|u\|_{H^s}^2$ for $(*)$. Now we turn to the second term $(**)$, which fulfills
	\begin{equation*}
		|(**)| = |(\Lambda^su, \Lambda^sf(u))| \leq \|u\|_{H^s}\|\Lambda^sf(u)\|_{L^2}.
	\end{equation*}
	It remains to show
	\begin{equation}\label{ee4}
		\|\Lambda^sf(u)\|_{L^2} \leq C(t)\|u\|_{H^s}.
	\end{equation}
	Recall the formulation of $f(u)$ in \eqref{f}, we have
	\begin{align*}
		\|\Lambda^sf(u)\|_{L^2} \lesssim \underbrace{\|\Lambda^{s-2k}([\Gamma, u^p]u_x)\|_{L^2}}_{(**_1)} +b \underbrace{\|\Lambda^{s-2k}(u^{p-1}u_x\Gamma u)\|_{L^2}}_{(**_2)}\\
		+ \underbrace{\|\Lambda^{s-2k}((g(u))_x)\|_{L^2}}_{(**_3)} + (b+1)\underbrace{\|\Lambda^{s-2k}(u^pu_x)\|_{L^2}}_{(**_4)}.
	\end{align*}
	We estimate each term separately:
	\begin{itemize}
		\item[i)] \textbf{Estimate of $(**_1)$}: We have
		\begin{align*}
			\|\Lambda^{s-2k}[\Gamma, u^p]u_x\|_{L^2} &= \|\Lambda^{s-2k}\big(\Gamma(u^pu_x) - u^p\Gamma(u_x)\big)\|_{L^2},\\
			&=\left\|\Lambda^{s-2k}\left(\sum_{j=0}^{k}\binom{k}{j}(-1)^j\partial_x^{2j}(u^pu_x) - \sum_{j=0}^{k}\binom{k}{j}(-1)^ju^p\partial_x^{2j+1}u\right)\right\|_{L^2},\\
			&\lesssim\ \sum_{j=0}^k\left\|\Lambda^{s-2k}\big(\partial_x^{2j}(u^pu_x) - u^p\partial_x^{2j+1}u\big)\right\|_{L^2}\\
			&\lesssim \sum_{j=1}^k\left\|\partial_x^{2j}(u^pu_x) - u^p\partial_x^{2j+1}u\right\|_{H^{s-2k}},
		\end{align*}
		where we used the binomial formula for the $\Gamma$ and the fact that when $j = 0$, $\partial_x^{2j}(u^pu_x) = u^p\partial_x^{2j+1}u$. Let's divide into cases:
		\begin{itemize}
			\item[(a)] For the case $1 \leq j \leq k-1$, the terms are lower order and can be controlled using Sobolev embeddings and the fact that $u\in H^s$ with $s>2k-\frac{1}{2}$. Each term involves at most $2j+1 \leq 2k-1$ derivatives, which is less than $s$, so the estimates are straightforward using product rules and the boundedness of $\|u\|_{H^{2k}}$ from the second conserved quantity. More precisely,
			\begin{align}\label{ee3}
				\left\|\partial_x^{2j}(u^pu_x) - u^p\partial_x^{2j+1}u\right\|_{H^{s-2k}} \lesssim \|u^pu_x\|_{H^{s-2k+2j}} + \|u^p\partial_x^{2j+1}u\|_{H^{s-2k}}.
			\end{align}
			If $s - 2k > 1/2$, then we estimate \eqref{ee3} further as
			\begin{align*}
				&\|u^pu_x\|_{H^{s-2k+2j}} + \|u^p\partial_x^{2j+1}u\|_{H^{s-2k}}\\
				&\lesssim \|u^{p+1}\|_{H^{s-2k+2j+1}} + \|u^p\|_{L^\infty}\|\partial_x^{2j+1}u\|_{H^{s-2k}} + \|u^p\|_{H^{s-2k}}\|\partial_x^{2j+1}u\|_{L^\infty}\\
				&\lesssim C(\|u\|_{L^\infty})\|u\|_{H^{s-2k+2j+1}} + \|u\|_{H^{s-2k+2j+1}} + \|u\|_{W^{2j+1,\infty}}C(\|u\|_{L^\infty})\|u\|_{H^{s-2k}}\\
				&\lesssim C(t)\|u\|_{H^s}
			\end{align*}
			where we used $\|u\|_{W^{2j+1,\infty}} \lesssim \|u\|_{H^{2j+2}} \lesssim \|u\|_{H^{2k}} \le C(t)$. If $1/2 \ge s - 2k$, we choose $\alpha = s - 2k$ and $\beta = 1/2 + \delta$ for a sufficiently small $\delta>0$. It can be easily checked that $\beta \ge \alpha > -\beta$, and thus we can apply the fractional Leibniz rule in Lemma \ref{leibniz} with $f = u^p$ and $g = \partial_x^{2j+1}u$ to estimate
			\begin{align*}
				\|u^p \partial_x^{2j+1}u\|_{H^{s-2k}} &\lesssim \|u^p\|_{H^{1/2+\delta}}\|\partial_x^{2j+1}u\|_{H^{s-2k}}\\
				&\lesssim C(\|u\|_{L^\infty})\|u\|_{H^{1/2+\delta}}\|u\|_{H^{s-2k+2j+1}}\\
				&\lesssim C(t)\|u\|_{H^s}.
			\end{align*}
			From both two cases, we conclude the estimate for \eqref{ee3}
			\begin{equation*}
				\left\|\partial_x^{2j}(u^pu_x) - u^p\partial_x^{2j+1}u\right\|_{H^{s-2k}} \lesssim C(t)\|u\|_{H^s}.
			\end{equation*}
			\item[(b)] For the case $j = k$, we need to estimate 
			\begin{align*}
				T_k &:= \partial_x^{2k}(u^pu_x)-u^p\partial_x^{2k+1}u.
			\end{align*}
			By writing $\partial_x^{2k}(u^pu_x) = \partial_x(\partial_x^{2k-1}(u^pu_x))$ and
			\begin{equation*}
				u^p\partial_{x}^{2k+1}u = \partial_x\big[u^p \partial_{x}^{2k}u - \partial_x^{2k-1}u\partial_x(u^p)\big] + \partial_x^{2k-1}u\partial_{xx}(u^p)
			\end{equation*}
			we have
			\begin{align*}
				\|T_k\|_{H^{s-2k}} &\lesssim \|\partial_x^{2k-1}(u^pu_x) - u^p \partial_{x}^{2k}u + \partial_x^{2k-1}u\partial_x(u^p)\|_{H^{s-2k+1}}\\
				&\qquad + \|\partial_x^{2k-1}u\partial_{xx}(u^p)\|_{H^{s-2k}} \lesssim (C) + (D). 
			\end{align*}
			For the part $(D)$, we use Lemma \ref{est1}
			\begin{align*}
				(D) &\lesssim \|\partial_x^{2k-1}u\|_{L^\infty}\|\partial_{xx}(u^p)\|_{H^{s-2k}} + \|\partial_x^{2k-1}u\|_{H^{s-2k}}\|\partial_{xx}(u^p)\|_{L^\infty}\\
				&\lesssim \|u\|_{H^{2k}}\|u^p\|_{H^{s-2k+2}} + \|u\|_{H^{s-1}}\|u^p\|_{H^3}\\
				&\lesssim C(t)\|u\|_{H^s}
			\end{align*}
			since $\|u\|_{H^3} \lesssim \|u\|_{H^{2k}} \le C(t)$. For the term $(C)$ we use product rule and Lemma \ref{est1} again to estimate
			\begin{align*}
				(C) &\lesssim \left\|\sum_{j=0}^{2k-1}\binom{2k-1}{j}\partial_x^{2k-1-j}(u^p)\partial_x^{j+1}u - u^p\partial_x^{2k}u + \partial_x^{2k-1}u\partial_x(u^p)\right\|_{H^{s-2k+1}}\\
				&\lesssim \sum_{j=0}^{2k-2}\|\partial_x^{2k-1-j}(u^p)\partial_x^{j+1}u\|_{H^{s-2k+1}}\\
				&\lesssim \sum_{j=0}^{2k-2}\left(\|\partial_x^{2k-1-j}(u^p)\|_{L^\infty}\|\partial_x^{j+1}u\|_{H^{s-2k+1}} + \|\partial_x^{2k-1-j}(u^p)\|_{H^{s-2k+1}}\|\partial_x^{j+1}u\|_{L^\infty}\right) \\
				&\lesssim \sum_{j=0}^{2k-2}\left(\|u^p\|_{H^{2k-j}}\|u\|_{H^{s-2k+j+2}} + \|u^p\|_{H^{s-j}}\|u\|_{H^{j+2}} \right)\\
				&\lesssim C(\|u\|_{L^\infty})\|u\|_{H^{2k}}\|u\|_{H^s} \lesssim C(t)\|u\|_{H^s}.
			\end{align*}
			Therefore, we have
			\begin{equation*}
				(**_1) \lesssim C(t)\|u\|_{H^s}.
			\end{equation*}
		\end{itemize}
		\item[ii)] \textbf{Estimate of $(**_2)$}:     We expand $\Gamma$ to have
		\begin{align*}
			\|\Lambda^{s-2k}(u^{p-1}u_x\Gamma u)\|_{L^2} &\lesssim \|u^{p-1}u_x\Gamma u\|_{H^{s-2k}},\\
			&\lesssim \Big\|\partial_x(u^p)\sum_{j=0}^{k}\binom{k}{j}(-1)^{j}\partial_x^{2j}u\Big\|_{H^{s-2k}},\\
			&\lesssim \sum_{j=0}^{k-1}\|\partial_x(u^p)\partial_x^{2j}u\|_{H^{s-2k}} + \|\partial_x(u^p)\partial_x^{2k}u\|_{H^{s-2k}},\\
			&= (E) + (F).
		\end{align*}
		Applying the fractional product estimate in Lemma \ref{est1} to $(E)$ yields
		\begin{align*}
			(E) &= \sum_{j=0}^{k-1}\|\partial_x(u^p)\partial_x^{2j}u\|_{H^{s-2k}}\\
			&\lesssim \sum_{j=0}^{k-1}\left(\|\partial_x(u^p)\|_{L^\infty}\|\partial_x^{2j}u\|_{H^{s-2k}} + \|\partial_x(u^p)\|_{H^{s-2k}}\|\partial_x^{2j}u\|_{L^\infty}\right)\\
			&\lesssim \sum_{j=0}^{k-1}\left(\|u^p\|_{H^2}\|u\|_{H^{s-2k+2j} + \|u^p\|_{H^{s-2k+1}}\|u\|_{H^{2j+1}}} \right)\\
			&\lesssim C(\|u\|_{L^\infty})\|u\|_{H^{2k}}\|u\|_{H^s} \lesssim C(t)\|u\|_{H^s}.
		\end{align*}
		It remains to estimate $(F)$, where we exploit
		\begin{equation*}
			\partial_x(u^p)\partial_x^{2k}u = \partial_x[\partial_x(u^p)\partial_x^{2k-1}u] - \partial_{xx}(u^p)\partial_x^{2k-1}u
		\end{equation*}
		and Lemma \ref{est1} to get
		\begin{align*}
			(F) &\lesssim \|\partial_x(u^p)\partial_x^{2k-1}u\|_{H^{s-2k+1}} + \|\partial_{xx}(u^p)\partial_x^{2k-1}u\|_{H^{s-2k}}\\
			&\lesssim \|\partial_x(u^p)\|_{L^\infty}\|\partial_x^{2k-1}u\|_{H^{s-2k+1}}+\|\partial_x(u^p)\|_{H^{s-2k+1}}\|\partial_x^{2k-1}u\|_{L^\infty}\\
			&\quad + \|\partial_{xx}(u^p)\|_{L^\infty}\|\partial_x^{2k-1}u\|_{H^{s-2k}} + \|\partial_{xx}(u^p)\|_{H^{s-2k}}\|\partial_x^{2k-1}u\|_{L^\infty}\\
			&\lesssim \|u^p\|_{H^2}\|u\|_{H^{s}} + \|u^p\|_{H^{s-2k+2}}\|u\|_{H^{2k}} + \|u^p\|_{H^{3}}\|u\|_{H^{s-1}} + \|u\|_{H^{s-2k+2}}\|u\|_{H^{2k}}\\
			&\lesssim C(\|u\|_{L^\infty})\|u\|_{H^{2k}}\|u\|_{H^s} \lesssim C(t)\|u\|_{H^s}.
		\end{align*}
		Combining the estimates of $(E)$ and $(F)$, we obtain
		\begin{equation*}
			(**_2) \le C(t)\|u\|_{H^s}.
		\end{equation*}
		\item[iii)] \textbf{Estimates of $(**_3)$ and $(**_4)$}. These are straightforward thanks to
		\begin{equation*}
			\|\Lambda^{s-2k}(g(u))_x\|_{L^2} \leq C\|g(u)\|_{H^{s-2k+1}} \leq C(\|u\|_{\infty})\|u\|_{H^s} \leq C\|u\|_{H^s},
		\end{equation*}
		by Lemma \ref{composition}, and the same for $(**_4)$.
	\end{itemize}
	Combining all these estimates, we get the desired estimate \eqref{ee4}. Thus, we have 
	\begin{equation*}
		\frac{d}{dt}\|u\|_{H^s}^2 \le C(t)\|u\|_{H^s}^2
	\end{equation*}
	which yields, due to Gronwall's lemma
	\begin{equation*}
		\|u(t)\|_{H^s}^2 \le \exp\Big(\int^t C(r)dr\Big)\|u_0\|_{H^s}^2
	\end{equation*}
	which concludes the global existence.
\end{proof}

\medskip
\section*{Acknowledgement}
Part of this work is completed during the visits of the second author to the University of Graz and the third author to Sabanc\i\,  University, and hospitality of the universities is greatly acknowledged. This work is supported by NAWI Graz.

%\nocite{*}

\end{document}